\documentclass[11pt]{amsart}

\usepackage[english]{babel}
\usepackage[letterpaper,top=3cm,bottom=3cm,left=3cm,right=3cm,marginparwidth=1.75cm]{geometry}
\usepackage{amsmath,amssymb,amsthm}
\usepackage{graphicx}
\usepackage[colorlinks=true, allcolors=blue]{hyperref}
\usepackage{xcolor}
\usepackage{comment}
\usepackage{enumerate}

\usepackage{tikz}
\usetikzlibrary{trees,calc}

\theoremstyle{definition}
\newtheorem{definition}{Definition}[section]

\newtheorem{remark}[definition]{Remark}

\theoremstyle{plain}
\newtheorem{theorem}{Theorem} 
\newtheorem{lemma}[definition]{Lemma} 
 
\newtheorem{proposition}[definition]{Proposition}

\numberwithin{equation}{section}

\newcommand\pa{\partial}
\newcommand\Ga{\Gamma}
\newcommand\ga{\gamma}
\newcommand\gm{\gamma}
\newcommand\Om{\Omega}

\title{An inverse problem for semilinear wave equations on metric tree graphs}

\author{Sergei Avdonin}
\address{Department of Mathematics and Statistics, University of Alaska at Fairbanks, Fairbanks, Alaska, USA, and Moscow Center of Fundamental and Applied Mathematics of Lomonosov Moscow State University }
\email{saavdonin@alaska.edu}

\author{Matti Lassas}
\address{Department of Mathematics and Statistics, University of Helsinki, Helsinki, Finland}
\email{matti.lassas@helsinki.fi}

\author{Jinpeng Lu}
\address{Department of Mathematics and Statistics, University of Helsinki, Helsinki, Finland}
\email{jinpeng.lu@helsinki.fi }

\author{Medet Nursultanov}
\address{Department of Mathematics and Statistics, University of Helsinki, Helsinki, Finland}
\email{medet.nursultanov@gmail.com}

\author{Lauri Oksanen}
\address{Department of Mathematics and Statistics, University of Helsinki, Helsinki, Finland}
\email{lauri.oksanen@helsinki.fi}


\begin{document}
\maketitle

\vspace{-5mm}

\begin{abstract}
We study the inverse problem for a semilinear wave equation on metric tree graphs. From the Dirichlet-to-Neumann map defined at all but one of the boundary vertices, we recover unknown connectivity of the graph, lengths of the edges, the time-independent potential and the time-dependent coefficient of the nonlinear term of the equation.
\end{abstract}

\section{Introduction}

Let $\Omega=\Omega(V,E)$ be a finite connected metric graph with the set of vertices $V$ and the set of edges $E$. 
We consider the following semilinear wave equation on the metric graph $\Omega$ up to time $T>0$,
%with real-valued time-independent potential $q$ and time-dependent coefficient function $a$:
\begin{equation}\label{eq-wave-nonlinear-general}
    \begin{cases}
    (\partial_{t}^2 -\partial_x^2 ) u(x,t)+q(x,t) u(x,t)+a(x,t)u^3 (x,t)=0, &x\in \Omega\setminus V, \ t \in (0,T),\\
    u(x,t)|_{t \leq 0}=0, & x\in \Omega,
    \end{cases}
\end{equation}
with given boundary conditions on a subset of vertices $\Gamma\subset V$ and matching conditions on the other vertices $V\setminus \Gamma$. 
We are interested in the following inverse problem.

\smallskip
\emph{Inverse Problem.}
Given the Dirichlet-to-Neumann map for the equation \eqref{eq-wave-nonlinear-general} on a subset of vertices $\Gamma\subset V$ up to time $T$, is it possible to determine the graph structure (i.e., the connectivity and lengths of edges) and the functions $a,q$?

\smallskip
Clearly one can only hope to solve this problem up to isometry of the graph, and only recover the functions $a,q$ in some time domain depending on $T$.
%Outside of such examples, counterexample is not known to us.
As the first step, we study this problem on metric trees,
i.e. graphs without cycles.
The goal of this paper is to solve the inverse problem for metric trees with the Dirichlet-to-Neumann map defined on the set of all but one of the leaves, in the case that the potential $q$ is time-independent.

\subsection{Inverse problems for linear and nonlinear equations on metric graphs.}  Inverse theory of network-like structures is an important part of the rapidly developing area of
applied mathematics -- analysis on graphs with a variety of important applications to many problems of science and engineering. The known results in this direction concern almost exclusively linear differential equations defined on trees and periodic lattices.

A variety of inverse spectral problems on metric graphs were studied 
in many papers, see e.g. books \cite{BerkolaikoKuchment}, \cite{Kur2023}, survey \cite{Yurko2016} and references therein. Dynamical inverse problems on metric trees were studied, e.g. in \cite{AK08, avbell15, AE}.
%A discrete version of the inverse conductivity problem on tree graphs was studied in \cite{GR}.
  In particular, the most significant result of \cite{AK08} developed a constructive and robust procedure for recovering tree's parameters, which became known as the  leaf-peeling  (LP)  method.
  This method was extended to inverse boundary problems for various types of partial differential equations (PDEs) on trees
  in a series of subsequent papers, see, e.g. \cite{avbell15, ACLM, AMN, AZ21}.
  The proposed identification procedure is  recursive, and it allows
  recalculating efficiently the inverse data from the
  original tree to the smaller trees, ``pruning'' leaves step by step up to the rooted edge. 
  A discrete version of the inverse conductivity problem was studied on tree graphs \cite{GR} and on (multi-dimensional) square lattices \cite{CM,DJ}.
  Inverse problems on (perturbed) periodic lattices are studied in e.g. \cite{Ando13,AIM,ABE,IK,IM,LGGY,WYB}, and recently in \cite{BEILL} based on a different approach developed in \cite{BILL,BILL2}.
  
Nonlinear equations on graphs naturally appear in several fields of sciences, such as condensed matter physics, hydrodynamics, nonlinear waveguides and optic fibers, see e.g. \cite{Adami,Sobirov} and references therein.
  Initial boundary value problems and control problems for nonlinear PDEs on metric graphs were studied in several works, see e.g. \cite{mehmeti1994nonlinear,GLL}.  
To the best of our knowledge, we are not aware of papers concerning inverse problems for nonlinear differential equations on metric graphs. In the present paper we study inverse problems for nonlinear equations on trees, and plan to consider general metric graphs in subsequent papers. We use the ideas for solving inverse problems for nonlinear PDEs on manifolds developed in \cite{KLU} and
  the ideas of leaf-peeling method or, more exactly, its ``leaf-cleaning'' version \cite{avdonin2015source}.

\subsection{Main result}

Let $\Omega=\Omega(V,E)$ be a finite, connected graph, where $V,E$ are the sets of vertices and edges of $\Omega$ respectively.  Each edge in $E$ is associated with two vertices in $V$ called its endpoints.  We assume that $V = \{v_i: i \in I\}$ and $E= \{e_j: j \in J\}$ are nonempty;  $I$ and $J$ are subsets of $\mathbb{N}.$ 
An edge $e_j$ between $v_i$ and $v_k$ is denoted as $e_j(v_iv_k)$.  The set of indices of the edges  incident to $v_i$ is denoted by $J(v_i)$.  For trees,  the set of all leaves (vertices of degree $1$), denoted by  $ \Gamma:=\{v_i \in  V: |J(v_i)|=1\}, $  plays the role of the graph boundary.
Here, we use $|\cdot|$ to denote the cardinality of a set of features on a graph. 
We recall that a graph is called a {\bf metric graph} if every edge $ e_j \in E$  is identified with an interval $(0,l_j) $ of the real line with a positive length $l_j$.

To describe our result on inverse problem for trees, we assume that controls and observations are available at all leaves in $\Gamma$ except one leaf $\gamma_0$ where the Dirichlet condition is imposed.
Namely, we set 
\begin{equation}
\Gamma_0:=\Gamma\setminus \{\gamma_0\},
\end{equation}
and consider the equation
\begin{equation}\label{eq-wave-nonlinear}
    \begin{cases}
    (\partial_{t}^2 -\partial_x^2 ) u(x,t)+q(x) u(x,t)+a(x,t)u^3(x,t)=0, &x\in \Omega\setminus V, \ t \in (0,T),\\
    u(\cdot,t)|_{\Gamma_0}=f(t),\quad u(\gamma_0,t)=0 & t \in (0,T),\\
    %u(\gamma_0,t)=0, & \gamma_0\in \Gamma, \, t\in (0,T),\\
    u(x,t)|_{t \leq 0}=0, &  x\in \Omega, \\
%          u(x,0) = 0, \ \partial_t u(x,0) = 0, &x\in \Omega,\\
          u \textrm{ subject to the Kirchhoff-Neumann condition on } V\setminus\Gamma.
    \end{cases}
\end{equation}
The Kirchhoff-Neumann condition states the continuity of the solution $u$  at each vertex 
and
that the sum of the normal derivatives of $u$  in the direction outwards the vertex is zero; a precise formulation is given in \eqref{s-d}.
Note that the Kirchhoff-Neumann condition at the vertices of degree $1$ reduces to the usual Neumann condition.

We introduce the Dirichlet-to-Neumann map on $\Gamma_0$ by the rule
\begin{equation}\label{DN-nonlinear}
\Lambda_T f = \partial u^f|_{\Gamma_0 \times (0,T)},
\end{equation}
where $u^f$ is the solution of the problem \eqref{eq-wave-nonlinear}.
Note that we will only use the Dirichlet-to-Neumann map for small sources $f$ (in certain norms, depending on $T$), in which case the semilinear equation \eqref{eq-wave-nonlinear} has a unique solution, see Lemma \ref{lemma_for_direct_problem}.
Observe that the wave observations cannot distinguish vertices of degree $2$, at which the wave behavior is not altered by the Kirchhoff-Neumann condition at the vertex.

We write $C^1(\Omega\times[0,T])$ for the space of functions that are continuous on $\Omega\times[0,T]$, and are $C^1$ on each edge-time rectangle $[0,l_j]\times[0,T]$ for every $e_j\in E$.
In \eqref{eq-wave-nonlinear}  we assume that
\begin{equation} \label{a-q-C2}
a\in C^1(\Omega\times[0,T]).
\end{equation}
We prove that the Dirichlet-to-Neumann map \eqref{DN-nonlinear} on $\Gamma_0$ up to time $T$ for the semilinear wave equation \eqref{eq-wave-nonlinear} determines the tree structure and coefficients $a,q$ up to some time, assuming $T$ is sufficiently large.

%The assumption $q\in C^1$ is needed for remainder estimate in geometric optics, see \eqref{eq-remainder-estimate}.
%The assumption $a\in C^1$ is needed in Section \ref{direct}.
\begin{theorem} \label{main-tree}
Let $\Omega$ be a finite metric tree and $\Gamma$ be the set of all leaves of $\Omega$. 
Let $\gamma_0$ be one leaf and denote $\Gamma_0=\Gamma\setminus \{\gamma_0\}.$ 
Suppose that we are given the Dirichlet-to-Neumann map $\Lambda_T$ on $\Gamma_0$ for the semilinear wave equation \eqref{eq-wave-nonlinear} up to time $T$, with $q\in C^1(\Omega)$ and $a\in C^1(\Omega \times [0,T])$.
Denote 
\begin{equation}\label{diameter-gamma0}
D(\gamma_0):=\max_{\gamma\in \Gamma_0} \, d (\gm,\, \gm_0 ),
\end{equation}
where $d$ is the distance function on the metric tree.
Assume $T>2D(\gamma_0)$. Then $\Lambda_T$ uniquely determines the metric tree structure of $\Omega$ up to vertices of degree $2$, the (time-independent) potential $q$ on $\Omega$, and the coefficient $a$ on
\begin{equation} \label{Domain_recovery}
\Big\{(x,t)\in \Omega\times [0,T]: D(\gamma_0)+d(x,\gamma_0) \leq t \leq T-D(\gamma_0)+d(x,\gamma_0) \Big\}.
\end{equation}
\end{theorem}

\begin{remark}
The domain \eqref{Domain_recovery} in Theorem \ref{main-tree} is not the maximal domain that the given data can recover, especially on the edges incident to the boundary vertices in $\Gamma_0$, but it is a simple formulation valid for the whole graph, see Figure \ref{fig-3-layers}.
%On the edges incident to the boundary vertices in $\Gamma_0$, the domains of recovery are in general much larger given by Proposition \ref{star-diff}.
For the simple case of interval $[0,L]$, the domain of recovery \eqref{Domain_recovery} for the coefficient $a$ is illustrated in Figure \ref{fig_domain_interval}.
%Effectively, the domain \eqref{Domain_recovery} is the same as the domain of recovery on the interval $[0,D(\gamma_0)]$.
\end{remark}

\begin{figure}[h]
\includegraphics[width=0.43\linewidth]{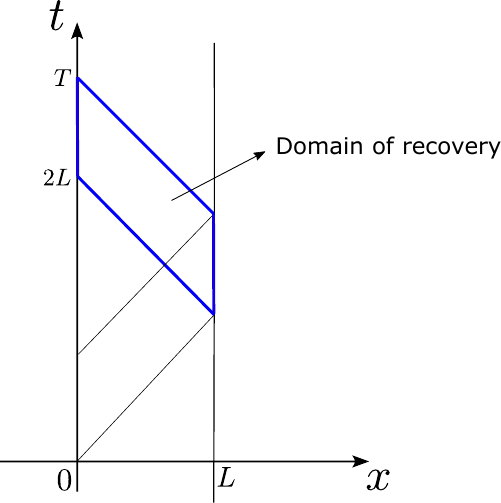}
\caption{As a special case of Theorem \ref{main-tree}, the reconstruction of the time-dependent coefficient $a$ on the interval $[0,L]$. Suppose control and observation of waves are available at only one endpoint $x=0$, i.e., $\Gamma_0=\{0\}$ and $\gamma_0=\{1\}$. Then $\Lambda_T$ uniquely determines $a$ on the domain \eqref{Domain_recovery}, the one enclosed by the blue lines.}
\label{fig_domain_interval}
\end{figure}

 Since the first linearization of $\Lambda_T$ is the Dirichlet-to-Neumann map for the linear equation $(\partial_t^2-\partial_x^2)v+qv=0$ (see Lemma 2.2 below), the first part of Theorem \ref{main-tree} (recovery of the metric tree structure and time-independent potential) follows from known results, e.g. \cite{AK08, AZ21}. Thus, we assume the metric tree structure and $q$ are known from now on and focus on the second part of the theorem: the recovery of the time-dependent coefficient $a$. 
 The basic idea is to use the linearization method developed in \cite{KLU} for manifolds. 
The main difficulty in the discrete case is that waves can interact with their reflections at the internal vertices in a complicated way due to general configurations of the graph, which may cause certain information to be hidden from wave observations.
In the case of metric trees, we are able to construct a general procedure in this paper so that each edge can be isolated from wave observations.
 We note that our method also works for nonlinearity $u^m$ with $m\geq 3$.

\smallskip
 The paper is organized as follows. In Section \ref{prel} we introduce metric graphs and discuss an inverse problem for the linear wave equation.
In Section \ref{sec-CGO} we study the geometric optic solutions for the linear wave equation on metric graphs. Section \ref{sec-linearization} is devoted to linearization techniques and presents the solution on our inverse problem for the simplest graphs. Section \ref{sec-general} completes the solution to the inverse problem for the semilinear wave equation on metric trees. 
Finally, Section \ref{direct} contains the results concerning the direct problem for the semilinear wave equation on metric graphs.

\section{Preliminaries and inverse problems for linear wave equation} \label{prel}
\subsection{Metric graphs and Hilbert spaces on graphs}

The graph $ \Omega $ determines naturally the Hilbert space of square
integrable functions $ \mathcal{H}=L^2 (\Omega). $ We define the space  $\mathcal{H}^1$
of  continuous functions $y$ on $\Omega$ such that
$y_j:=y|_{e_j} \in H^1(e_j)$ for every $j \in J$ and denote by  $\mathcal{H}^1_0$ the subspace $\{y \in  \mathcal{H}^1 \,: y|_{\Gamma}=0\}.$ 
Let $\mathcal{H}^{-1}$ be the dual space of $\mathcal{H}^1_0$.
We further introduce the space $ \mathcal{H}^2$ of functions $y \in \mathcal{H}^1$  such that
$y_j \in H^2(e_j)$ for every $j \in J,$ and
for every $v \in V \setminus \Gamma$, the equalities
\begin{equation}\label{s-d}
\sum_{j \in J(v)} \partial y_j(v)=0; \quad y_j(v)=y_k(v), \ j,k \in J(v),
\end{equation}
holds. Here (and everywhere later on)
$\partial y_j(v)$ denotes the
derivative of $y$ at the vertex $v$ taken along the  edge  $e_{j}$
in the direction outwards the vertex.
Vertex conditions \eqref{s-d}  are called the {\bf standard}
or {\bf Kirchhoff-Neumann} (KN) conditions. For the boundary vertices they are just Neumann conditions.

Let $q$ be a real valued function (potential) such that $q|_{e_j} \in C[0,l_j]$. We define the Schr\"odinger operator on the graph $\Om$ as the operator 
\begin{equation} \label{sch}
\, \mathcal{L}  = - \frac{d^2}{dx^2} + q \,
\end{equation}
in $ \mathcal{H} $ with the domain $\mathcal{H}^2 \cap \mathcal{H}^1_0.$

 We now associate the following initial boundary value problem (IBVP) for the wave equation to the graph $\Omega$: 
 \begin{equation}\label{eq-wave-linear-KN}
    \begin{cases}
    (\partial_{t}^2 -\partial_x^2 ) u(x,t)+q(x) u(x,t)=0, &x\in \Omega \setminus V, \  t > 0,\\
    \sum_{j \in J(v)}  \partial u_j(v,t)=  0, & v \in V \setminus \Gamma, \ t > 0, \\ 
    u_j(v,t)-u_k(v,t)=0, &  j,k \in J(v), \ v \in V \setminus \Gamma, \  t > 0, \\
    u(\cdot,t)|_{\Gamma}=f(t), &  t > 0,\\
          u(x,0) = 0, \ 
      \partial_t u(x,0) = 0, &x\in \Omega.\\
    \end{cases}
\end{equation}
 This IBVP was studied by many authors.   
For different conditions on the potential $q$ 
the proof of the following proposition can be found in  \cite{mehmeti1994nonlinear}, \cite{lagnese1994analysis}, \cite{dager2006wave}, \cite{avdonin2015source},
and \cite{AZ}. 
\begin{proposition} \label{prop:unispace} 
	Let $m:=|\Gamma|, \, {f} \in \mathcal{F}^T: = L^2([0,T]; \mathbb{R}^m)$ and
	 $q \in L^1(\Omega).$ There exists a unique solution $	u =u^f(x,t)$ to the IBVP \eqref{eq-wave-linear-KN} (understood in a weak sense)  and
		\begin{equation} \label{inclusion}
u \in C([0,T]; \mathcal{H}) \cap C^1([0,T]; \mathcal{H}^{-1}),
	\end{equation}
		\noindent that is, $u(\cdot, t) \in \mathcal{H}$, $\partial_t u(\cdot, t) \in \mathcal{H}^{-1}$ for all $t \in [0,T]$, and both $u(\cdot, t)$ and $\partial_t u(\cdot, t)$ are continuous with respect to $t$ in  corresponding norms.
  If $f \in \mathcal{F}^T_1: = H^1([0,T]; \mathbb{R}^m), \, f(0)=0,$ the solution to the IBVP \eqref{eq-wave-linear-KN} $u(\cdot,t)$ belongs to the energy space for all $t \in [0,T]$: 
  \begin{equation} \label{energy}
u \in C([0,T]; \mathcal{H}^1) \cap C^1([0,T]; \mathcal{H}).
	\end{equation}
  If ${f} \in C^2([0,T]; \mathbb{R}^m)$, ${f}(0)={f}'(0)={0}$ and $q \in C(\Omega)$, the IBVP \eqref{eq-wave-linear-KN} has a unique classical solution. 
\end{proposition}

\subsection{Inverse dynamical and spectral problems for the IBVP \eqref{eq-wave-linear-KN}}
The {\it response operator} (dynamical Dirichlet-to-Neumann map) $R^T$ acting in $\mathcal{F}^T$ for the IBVP \eqref{eq-wave-linear-KN} is defined by the following rule:
 \begin{equation} \label{response}
(R^Tf)(t)=\partial u^f(\cdot,t)|_{\Ga}, \ \ t \in (0,T), \ \ Dom \; R^T = \mathcal{F}^T_1.
 \end{equation}
 The inverse dynamical problem is to recover the tree (connectivity of the graph $\Om$ and lengths of the edges) and the potential $q$ on the edges from the response operator $R^T$ and
determine  $T$ that allows to solve this problem.

Along with the dynamical problems, inverse spectral problems have also attracted great attention. 
Inverse spectral data is the sequence of pairs $\mathcal{S}:=\{\lambda_n,
\partial \phi_n|_{\Ga}\},\, n \in \mathbb{N},$ consisting of the eigenvalues of the operator $\mathcal{L}$
and traces of the derivatives of its eigenfunctions on the boundary of the graph. Another kind of spectral data, which is equivalent to $\mathcal{S},$ is the Titchmarsh--Weyl (TW) matrix function (spectral Dirichlet-to-Neumann map). 

The relations between different kinds of inverse data were discussed in several papers, see e.g. \cite{KKLM}, \cite{AK08}. In particular, it is known that the spectral data determine the response operator $R^T$ for all $T \geq 0$ through the Fourier--Laplace transform and vice versa. 
If the dynamical data are given on a finite time interval, we are not able to find the spectral data in such a way and we need to use dynamical methods for solving the inverse problem.

\smallskip
The following lemma allows us to solve the problem of finding $q(x)$ independently of the nonlinear problem. 
Let $\Gamma_0\subset \Gamma$, and let $\Lambda_T^{\rm lin}$ be the Dirichlet-to-Neumann map on $\Gamma_0$ for the linear equation, i.e., corresponding to the equation \eqref{eq-wave-nonlinear} with $a=0$.
\begin{lemma}
    The operator $\Lambda_T^{\rm lin}$ is determined by $\Lambda_T$.
\end{lemma}
\begin{proof}
    Let $f$ be a boundary source. For $\varepsilon>0$, consider the equation
    \begin{equation}\label{non_lin_eq_with_eps}
    \begin{cases}
    (\partial_{t}^2 -\partial_x^2 ) u(x,t)+q(x) u(x,t)+a(x,t)u^3(x,t)=0, &x\in \Omega\setminus V, \ t \in (0,T),\\
    u(\cdot,t)|_{\Gamma}=\varepsilon f(t), & t\in (0,T),\\
    u(x,t)|_{t \leq 0}=0, & x\in \Omega, \\
%          u(x,0) = 0, \ \partial_t u(x,0) = 0, &x\in \Omega,\\
          u \textrm{ subject to the Kirchhoff-Neumann condition on } V\setminus\Gamma.
    \end{cases}
\end{equation}
Let $u_\varepsilon$ be the solution of the equation. 
Using arguments similar to those in \cite[Proposition 2]{LassasLiimatainenPotenciano-MachadoTyni}, one can show that the solution is differentiable with respect to $\varepsilon$, see Proposition \ref{prop-diff-epsilon}. Define 
\begin{equation}    \Tilde{u}:=\left.\partial_\varepsilon u_\varepsilon \right|_{\varepsilon = 0}.
\end{equation}
By differentiating \eqref{non_lin_eq_with_eps} with respect to $\varepsilon$, we obtain that $\Tilde{u}$ solves the equation
\begin{equation}\label{non_lin_eq_with_eps1}
    \begin{cases}
    (\partial_{t}^2 -\partial_x^2 ) \Tilde{u}(x,t)+q(x) \Tilde{u}(x,t)=0, &x\in \Omega\setminus V, \ t \in (0,T),\\
    \Tilde{u}(\cdot,t)|_{\Gamma}= f(t), & t\in (0,T),\\
    \Tilde{u}(x,t)|_{t \leq 0}=0, & x\in \Omega,\\
%          \Tilde{u}(x,0) = 0, \ \partial_t v(x,0) = 0, &x\in \Omega,\\
          \Tilde{u} \textrm{ subject to the Kirchhoff-Neumann condition on } V\setminus\Gamma.
    \end{cases}
\end{equation}
Then $\Lambda_T^{\rm lin}$ can be expressed as follows
\begin{equation*}
    \Lambda_T^{\rm lin} f = \left.\partial \Tilde{u} \right|_{\Gamma_0\times (0,T)} = \partial_\varepsilon \big( \Lambda_T(\varepsilon f)\big) \big|_{\varepsilon = 0}. \qedhere
\end{equation*}
\end{proof}

\section{Geometric optics on metric graphs}
\label{sec-CGO}

This section is for the explicit formula, modulo $O(h)$, for the geometric optics solution of the (linear) wave equation on metric graphs with time-independent potential function $q$. We use this solution in the next section to find the coefficient $a(x,t).$
\begin{equation}\label{eq-gauss}
    \begin{cases}
    (\partial_{t}^2 -\partial_x^2 ) v(x,t)+q(x) v(x,t)=0, &x\in \Omega \setminus V, \  t > 0,\\    v(\cdot,t)|_{\Gamma}=f(t), & t> 0,\\
          v(x,t)|_{t\leq 0}=0, &x\in \Omega,\\
    v \textrm{ subject to the Kirchhoff-Neumann condition on } V\setminus\Gamma.
    \end{cases}
\end{equation}
The boundary source of our interest is of geometric optics type:
\begin{equation*}
\exp \big(i h^{-1} \theta(t) \big) \chi_b(t),
\end{equation*}
where $h>0$ is a small parameter, and $\chi_b$ is a smooth cutoff function near $b$ with an independent parameter $b>0$, say $\chi_b=1$ on $[b/2,3b/2]$ and vanishing outside of $(0,2b)$.
%The cutoff function $\chi_b$ is shifted a little so that the initial condition of the wave equation is satisfied.
To simplify computations, we choose 
\begin{equation} \label{theta}
\theta(t)=t,
\end{equation}
and the boundary source
\begin{eqnarray} \label{source}
f_h(t) := e^{ih^{-1}t} \chi_b(t).
\end{eqnarray}

The geometric optics solution for the wave equation on the interval $[0,l]$ with the boundary source \eqref{source} at $x=0$ has the following form up to time $l$ (see, e.g. \cite[Thm. 2.64]{KKL} or \cite{NO,Salo} for details),
\begin{equation}\label{eq_v}
    v(x,t)=e^{ih^{-1}(t-x)}\chi_b(t-x) +R_h(x,t), \quad x,t\in [0,l],
\end{equation}
where the remainder $R_h$ satisfies
\begin{equation} \label{eq-remainder-estimate}
|R_h(x,t)|\leq C_b h,
\end{equation}
for some constant $C_b>0$ depending on $\chi_b$ and $\|q\|_{C^1}$.
%The constant $C_b$ depends on $\|q\|_{C^1}$. First, the first-order antasz $a_1$ in basically the integral of $q$. The remainder estimate depends on the $L^{\infty}$-norm (or $L^2$-norm) of $(\partial_t^2-\partial_x^2+q)a_1$, so depends on $C^1$-norm (or Lipschitz norm) of $q$, and also $C^1$-norm of $\chi_b$, and also on $l$ (upper bound for time domain).
%This is by taking antasz to the first order $a=a_0+h a_1+...$, and its remainder has estimate $O(h)$ (the same order as the order of antasz, see \cite{NO,Salo}). The function $a_1$ is an integral of $q \chi$ so uniformly bounded. We assumed $q\in C^0$. So absorbing $h a_1$ to the remainder estimate gives $a_0 +O(h)$.
%One gets immediately $H^1$-estimate for remainder $O(h)$. In our 1-d case, this gives $C^0$ pointwise estimate by Sobolev embedding.

\begin{lemma} \label{go-reflect}
With the Dirichlet boundary condition imposed at $x=l$, the geometric optics solution $v(x,t)$ on the interval $[0,l]$, with the boundary source \eqref{source} at $x=0$, has the following form up to time $2l$,
\begin{equation*}\label{eq_v_reflected}
    v(x,t)=e^{ih^{-1}(t-x)}\chi_b(t-x) -e^{ih^{-1}(t+x-2l)}\chi_b(t+x-2l) +O(h), \quad x\in [0,l], \;t\in [0,2l].
\end{equation*}
\end{lemma}

To prove the lemma we check that the main term satisfies the equation and boundary conditions and estimate the rest. The extension of this lemma to tree graphs is discussed in the next subsection.

\subsection{Geometric optics on a star graph}
Let us consider a star graph, namely $n$ intervals $I_1,\cdots,I_n$ with length $l_1,\cdots,l_n$ glued at $x=0$, where we parametrize the intervals $I_j$ as $[0,l_j]$.
We a boundary source of the form \eqref{source} at time $t=-l_1$ at $x=l_1,$ the endpoint $z_1$ of $I_1$, namely $f(z_1,t)=f_h(t)$, and $0$ on all other boundary vertices. 
The solution $v^h$ of the equation \eqref{eq-gauss} with this boundary source has the following form on $I_1$,
\begin{equation}\label{v_h}
    v^h(x,t)=e^{ih^{-1}(t+x)}\chi_b(t+x) +O(h), \quad x\in I_1,\; t\in [-l_1,0],
\end{equation}
Note that the phase function here differs slightly from \eqref{eq_v} due to different parametrization.

The following lemma characterizes the behaviour of geometric optics solutions past the common vertex of a star graph.

\begin{lemma} \label{CGO-Y}
On the metric graph of gluing $n$ intervals at one common point parametrized as above, the geometric optics solution of the equation \eqref{eq-gauss}, with the boundary source chosen as above, has the following form on the time interval $t\in [0,\min_{j} l_j]$:
\begin{equation}
    v_1^h(x,t)= e^{ih^{-1}(t+x)}\chi_b(t+x) - \frac{n-2}{n} e^{ih^{-1}(t-x)}\chi_b(t-x) +R_h(x,t), \quad x\in I_1, 
\end{equation}
\begin{equation}
    v_j^h(x,t)= \frac{2}{n} e^{ih^{-1}(t-x)}\chi_b(t-x) +R_h(x,t), \quad x\in I_j,\  j=2,\cdots, n,
\end{equation}
where $v_j^h$ is the restriction of the solution $v^h$ on $I_j$.
\end{lemma}

\begin{proof}
By the Kirchhoff-Neumann condition, we have at $x=0$ for any $t$,
\begin{equation*}
    \begin{cases}
v_1^h(0,t)=v_2^h(0,t)=\cdots=v_n^h(0,t), \\
\sum_{j=1}^n \partial_x v_j^h(0,t)=0.
\end{cases}
\end{equation*}
Solving these equations and examining $0$-th order terms (with respect to $h$) gives the coefficients appeared in the formulae. Indeed, let $c_j$ be the coefficients appeared in $v_j^h$, respectively. Then the conditions above give
\begin{equation*}
    \begin{cases}
f_h(t)-c_1 f_h(t)=c_2 f_h(t)=\cdots=c_n f_h(t), \\
f_h'(t)+c_1 f_h'(t)-\sum_{j=2}^n c_j f_h'(t)=0.
\end{cases}
\end{equation*}
Hence 
\begin{equation*}
c_1=\frac{n-2}{n}, \quad c_j=\frac2n,\quad j=2,\cdots,n. \qedhere
\end{equation*}
\end{proof}

Now we consider the simplest graph, interval $[0,l]$ with the boundary control $f=f_h(t)$ applied at the left endpoint: $v(0,t)=f(t), \ v(l,t)=0.$ One can check that
for $t>0$ the solution of \eqref{eq-gauss} is presented as
\begin{equation} \label{intl}
v(x,t)= \Big(1+R_h(x,t) \Big)\,\Big( f(t-x)-f(t+x-2l)+f(t-x-2l)-f(t+x-4l)+\cdots \Big).
\end{equation}
If the boundary control $f=f_h(t)$ is applied at the right endpoint: $v(0,t)=0, \ v(l,t)=f(t),$ then the solution of \eqref{eq-gauss} takes the form
\begin{equation} \label{intr}
v(x,t)=\Big(1+ R_h(x,t) \Big)\,\Big(f(t+x-l)-f(t-x-l)+f(t+x-3l)-f(t-x-3l)+\cdots \Big).
\end{equation}
Lemma \ref{CGO-Y} and formulas \eqref{intl}, \eqref{intr} allow us to construct the solution to the problem \eqref{eq-gauss} with $f=f_h(t)$ for all $t>0.$

To derive the equations describing the solution to the problem \eqref{eq-gauss} on a general graph we can propose another way to solve the problem described in Lemma \ref{CGO-Y}. We denote  by $g(t)$  the value of the solution $v(0,t)$ at the interior vertex and write the matching conditions connecting $g$ and $f$.
We obtain a delay integral equation of the second kind Volterra type. On a general graph we obtain a system of such equations.

Let us begin with a star graph described above. We put
$$F(x,t)=f(t+x-l_1)-f(t-x-l_1)+f(t+x-3l_1)-f(t-x-3l_1)+\ldots,$$
$$G_j(x,t)=g(t-x-l_1)-g(t+x-l_1-2l_j)+g(t-x-l_1-2l_j)-g(t+x-l_1-4l_j)+\ldots,\ \; j=1,\ldots,n.  $$
Up to the zero-th order terms,
$$v_1(x,t)=\Big(1+ R_h(x,t) \Big) \Big(F(x,t)+G_1(x,t)\Big), $$$$ v_j(x,t)=\Big(1+ R_h(x,t) \Big)G_j(x,t), \ \; j=2,\ldots,n.$$  
%Here 
%$$\psi_0(x,t)=\frac12 \int_{l_1}^x q_1(s) ds, \ \psi_j(x,t)=\frac12 \int_{0}^x q_1(s) ds + \alpha_j, \; j=1,\ldots,n,$$
%where constants $\alpha_j$ will be determined below. We may originally assume that $\alpha_j$ are piece wise constant functions of $t,$ but it turns out that all of them are equal to $\alpha$ indicated above and do not depend on $t$ or $j$. 
%Since $F(0,t)=0$ and $G_j(0,t)=g(t)$ for all $j$ an $t,$ the continuity condition at $x=0$ implies 
%\begin{equation} %\label{ajk}
%\alpha_j=\alpha_k, \; \ j,k=1,\ldots,n .
%\end{equation}
Now we present the second KN condition, $\sum_{j=1}^n \pa_x v_j(0,t)=0.$ It is convenient to introduce a new time variable instead of $t-l_1$ and denote it again by $t.$ Taking into account that
$$ \pa_x F(0,t)=2\Big(f'(t)+f'(t-2l_1)+f'(t-4l_1)+\ldots \Big),$$
%$$\pa_x G_j(0,t)=-[g'(t)+2g'(t-2l_j)+2g'(t-4l_j)+\ldots],$$
and examining the zero-th order term in $h$, we obtain:
\begin{equation} \label{kns}
\pa_x F(0,t)+\sum_{j=1}^n \pa_x  G_j(0,t)=0. %\alpha \, \pa_x  F(0,t)+\sum_{j=1}^n \alpha_j \pa_x  G_j(0,t)=0.
\end{equation}
%It follows that $\alpha_j=\alpha$ for all $j$ and $t.$
This equation can be presented in the form
\begin{equation} \label{delay}
2\Big( f'(t)+f'(t-2l_1)+f'(t-4l_1)+\ldots \Big)=\sum_{j=1}^n \Big(g'(t)+2g'(t-2l_j)+2g'(t-4l_j)+\ldots \Big),
\end{equation}
and rewritten as 
\begin{equation} \label{del}
2 \Big( f(t)+f(t-2l_1)+f(t-4l_1)+\ldots\Big)=\sum_{j=1}^n \Big(g(t)+2g(t-2l_j)+2g(t-4l_j)+\ldots \Big).
\end{equation}
It allows us to find $g(t)$ by $f(t)$ and, therefore, to find the main terms of the Gaussian beam on a star graph. For that, we consider all linear combinations $\sum_{j=1}^n 2\,k_j\,l_j$ with $k_j \in \{0,1,2,\ldots \}$ and
organize these numbers $0,t_1,t_2,\ldots$ in the increasing order on the real axis. Then the equations \eqref{delay} can be solved in steps starting with the interval $(0,t_1),$ then on $(t_1,t_2),$ and so on.

Example. Consider a 3-star graph with $l_1=2,\,l_2=1,\,l_3=5.$ Then solving equation \eqref{del} in this specific case we obtain:

for $0<t<2, \ g(t)=\frac23\,f(t);$

for $2<t<4, \ g(t)=\frac23\,f(t)-\frac49\,f(t-2);$

for $4<t<6, \ g(t)=\frac23\,f(t)-\frac49\,f(t-2)+\frac{2}{27}\,f(t-4);$

for $6<t<8, \ g(t)=\frac23\,f(t)-\frac49\,f(t-2)+\frac{2}{27}\,f(t-4)-\frac8{81}\,f(t-6);$ \\ and so on.

Now we will derive similar equations that allow to construct Gaussian beams on any metric graph. Suppose a finite metric graph $\Omega$ contains  $M$ vertices and $N$ edges. (At this point we denote vertices by $v_j$ and the solution to the wave equation by $u.$)
Let $e_k=(v_i, v_j)$ be an edge in $\Omega$ which is identified with the interval $[0,l]$.  Suppose $v_i$ is identified with $0$ and $v_j$ is identified with $l$ (we say $e_k$ goes from $v_i$ to $v_j$).  Define operators $W_k^\pm: L_2(0,T) \mapsto C([0,T];L_2(0,l))$: $$(W_k^-(f)) (x,t) = u_k^{f, v_i}(x,t), \quad \textrm{and} \quad  (W_k^+(f)) (x,t) = u_k^{f, v_j}(x,t).$$ 
 where $u_k^{f, v_i}$ and $u_k^{f, v_j}$ are obtained using \eqref{intl} and \eqref{intr} respectively.  

Operators $\partial^-$ and $\partial^+$ on $C([0,T]; L_2([0,l])$ are defined for taking derivative of the wave function along an edge outward of its vertices: $\partial^-(u) = \partial_x u (0, \cdot)$ and $\partial^+(u) = -\partial_x u (l, \cdot)$.  We have four combinations of $\partial^\pm$ and $W_k^\pm$ on $e_k$:
$$
(\partial^- W^-_k) (f) = -f'(t)   
-2f'(t-2l) 
-2f'(t-4l) +\dots \,,
$$
$$
(\partial^+ W^-_k) (f)= 2f'(t-l) +2w(l,l) f(t-l)
+2f'(t-3l)  +\dots \,,
$$
$$
(\partial^- W^+_k) (f)=2f'(t-l) 
+2 f'(t-3l) +\dots \,,
$$
$$
(\partial^+ W^+_k) (f)=-f'(t)  
-2f'(t-2l) 
-2f'(t-4l)  -\dots \,.
$$

We next define an $N \times M$ matrix operator $U,$ such that $U$ has one column for each vertex and one row for each edge.  The entries of $U$ are defined in analogy to the entries in the incident matrix of $\Omega$:  if there is an edge $e_k$ from $v_i$ to $v_j$, then $U_{k,i}=W_k^-$ and $U_{k,j}=W_k^+$.  All other entries in $U$ are zeros.  As one can see, $Ug$ gives us a column vector, where the $k^{th}$ entry equals to $u_k(x,t)$ on the edge $e_k$. If $g(t)$ is known then $Ug$ represents the solution for \eqref{eq-gauss}. 

Operator $K$ is defined as an $M \times N$ matrix operator. Its entries are defined in analogy to the transpose of the incident matrix of $\Omega$: if there is an edge $e_k$ from vertex $v_i$ to $v_j$, then $K_{ik}=\partial^-$ and $K_{jk}=\partial^+$.   All other entries in $K$ are zeros.  Now $KUg$ is a column vector of $M$ entries.  The $i^{th}$ entry represents the derivative sum of $u$ at a vertex $v_i$ (over all incident edges, outwards of $v_i$).  

Since the KN conditions hold only at the internal vertices,  we use an $M \times M$ diagonal matrix $D$ to pick out the interior vertices.  Let $D_{ij}=1$ if $i=j$ and $v_i \in V \setminus \Gamma$, $D_{ij}=0$ otherwise.  So the KN conditions on $V \setminus \Gamma$ can be represented by 
\begin{equation} \label{svi}
DKUg = 0\,.
\end{equation}
Equation \eqref{svi} is a system of $|V|-|\Gamma|$ equations similar to Equation \eqref{delay}:  
\begin{equation} \label{splitgroup}
\deg(v_i)\, g_i'(t) =F_i(t),
\end{equation}
 where for all $i=1, \dots, |V|-|\Gamma|$, $F_i(t)$ depends on the vector function $g$ with arguments delayed by some positive number. 
 The values of $g(t)$ on $\Gamma$ are given as boundary conditions in \eqref{eq-gauss}. Since $g(0)=0$, $g$ on $V \setminus \Gamma$ can be calculated in steps. 
 
%To compute $\psi_j(x),$ the first order terms in $h,$ on the edge $e_j$ we assume that our graph is tree, the control $f=f_h(t)$ is applied at one of the boundary vertices, say $v_1,$ and at all other  boundary vertices we have zero Dirichlet conditions. Let the edge $e_1(v_1,v_2)$ be incident to $v_1,$
%the edge $e_2(v_2,v_3)$ be incident to $v_2,$ and the edge $e_3(v_3,v_4)$ be incident to $v_3,$ and every $e_j(v_j,v_{j+1})$ is parametrized by $x \in (0,l_j)$ with $0$ at $v_j.$
%Then, as we demonstrated for a star graph,
 %$$\psi_1(x)=\int_0^x q_1(s)\,ds, \ \psi_2(x)=\int_0^x q_2(s)\,ds + \int_0^{l_1} q_1(s)\,ds, $$$$ \psi_3(x)=\int_0^x q_3(s)\,ds + \int_0^{l_2} q_2(s)\,ds +\int_0^{l_1} q_1(s)\,ds.  $$
 
% Since on a tree there is a unique path from $v_1$ to any other vertex, we can determine in such a way  $\psi_j(x)$ for every $e_i.$
 
 %{\bf S: I dropped here the factor $1/2$ of all integrals, it can be added later.
 %The result depends on our choice of the function $\theta$ and may have a slightly different form, but the approach works in any case.}

\section{Linearization method}
\label{sec-linearization}

Using the linearization method, the Dirichlet-to-Neumann map $\Lambda_T$ defined in \eqref{DN-nonlinear} determines
\begin{equation}
\partial_{\varepsilon_1} \partial_{\varepsilon_2} 
\partial_{\varepsilon_3}
\Big(\Lambda_T (\varepsilon_1 f_1+ \varepsilon_2 f_2+\varepsilon_3 f_3)\Big)  \Big|_{\varepsilon_1=\varepsilon_2=\varepsilon_3=0} = \partial w \big|_{\Gamma_0 \times (0,T)},
\end{equation}
where 
\begin{equation} \label{nonlinear_w}
\begin{cases}
(\partial_{t}^2-\partial_x^2) w +qw = -6a v_1 v_2 v_3, \\
w|_{\Gamma\times (0,T)}=0, \\
w|_{t\leq 0}=0,
\end{cases}
\end{equation}
and, for $j=1,2,3$,
\begin{equation}
\begin{cases}
(\partial_{t}^2-\partial_x^2) v_j +q v_j= 0, \\
v_j|_{\Gamma_0\times (0,T)}=f_j, \ v_j|_{\{\gamma_0\}\times (0,T)}=0, \\
v_j|_{t\leq 0}=0.
\end{cases}
\end{equation}
Let $T>0$, and $v_0$ be the solution of an auxiliary wave sending from $t=T$,
\begin{equation}
\begin{cases}
(\partial_{t}^2-\partial_x^2) v_0 +q v_0= 0, \quad\quad\quad t\in (0,T), \\
v_0|_{\Gamma_0\times (0,T)}=f_0,  \ v_0|_{\{\gamma_0\}\times (0,T)}=0, \\
v_0|_{t= T}=0, \quad \partial_t v_0|_{t= T}=0.
\end{cases}
\end{equation}
Using the equations above and integration by parts, one has
\begin{equation} \label{a-integral}
\int_0^T \int_{\Omega} 6av_0 v_1 v_2 v_3 = \int_0^T \int_{\Gamma_0} f_0 \partial w,
\end{equation}
the right-hand side of which is determined by the given data $\Lambda$.

\smallskip
For starters, we could consider the problem on the interval $[0,l]$ where we have control and observation only at $x=0$, with the Dirichlet condition at $x=l$. 
The wave $v_1$ is the geometric optics solution \eqref{eq_v} with the boundary source \eqref{source} sent at $t=t_0>0$ from $x=0$.
%Note that the requirement of $t_0>0$ is due to the initial condition of the equation \eqref{eq-wave-nonlinear}, and the size $b$ of the cut-off function in geometric optics \eqref{source} needs to chosen such that $b<t_0$.
%the geometric optics is nonzero near $t=t_0$.
The linear waves $v_2,v_3$ are sent from $x=0$ at some delayed time $t=s>t_0$. The wave $v_0$ is sent from $x=0$ at time $t=2l+t_0$. Then the wave $v_2,v_3$ and the reflected wave of $v_1$ (after hitting $x=l$) and the wave $v_0$ intersect at a small region of size $2b$, see Figure \ref{fig_interval}.
Assume $T>2l$.
As $h,b\to 0$, the left-hand side of \eqref{a-integral} converges to the pointwise value of $a$. 
Ranging the delay parameter $s\in [t_0,2l+t_0]$ recovers $a$ on the line $\{x+t=2l+t_0\}$. This calculation uses the explicit formula for the reflected wave of $v_1$ derived in Section \ref{sec-CGO}. 
Then ranging parameter $t_0\in (0,T-2l)$ recovers the time-dependent coefficient $a$ on a parallelogram domain.
The details are done in the following lemma.

%\begin{figure}
%\includegraphics[width=0.4\linewidth]{interval}
%\caption{Setting of Lemma \ref{interval} and \ref{star-equal}.}
%\label{fig_interval}
%\end{figure}

\begin{figure}
    \begin{minipage}{0.47\textwidth}
        \includegraphics[width=0.83\linewidth]{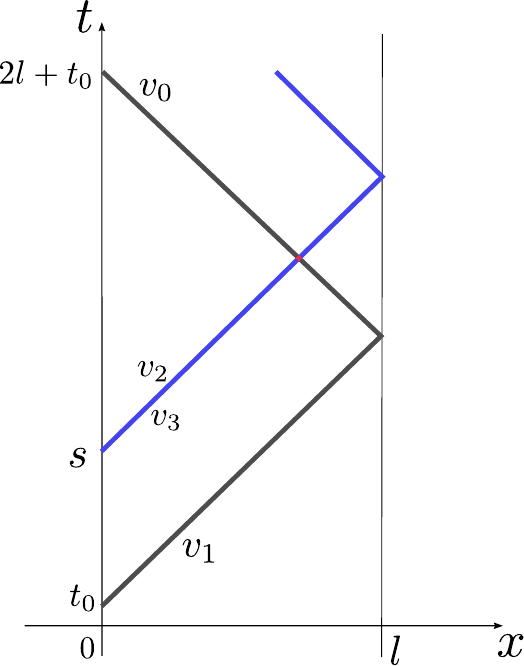}
        \caption{Setting of Lemma \ref{interval} and \ref{star-equal}.}
        \label{fig_interval}
    \end{minipage}
    \begin{minipage}{0.47\textwidth}
        \includegraphics[width=0.83\linewidth]{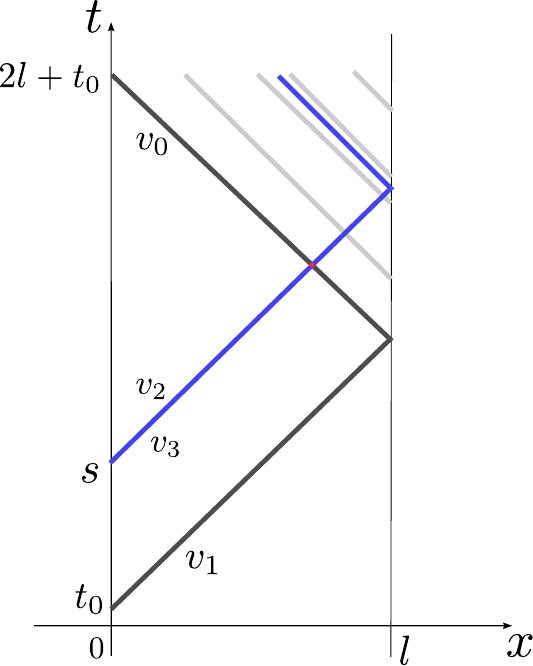}
        \caption{Setting of Proposition \ref{star-diff}.}
       \label{fig_diff}
    \end{minipage}
\end{figure}

\begin{lemma} \label{interval}
	Let $\Omega= [0,l]$, and suppose that control and observation are available at $x=0$ (i.e., $\Gamma_0=\{0\}$). 
 %and Dirichlet condition on $x=l$.
 Assume $T>2l$.
 Then $\Lambda_T$ uniquely determines $a$ {on $\big\{(x,t)\in [0,l]\times [0,T]: 2l \leq x+t \leq T \big\}$}.
\end{lemma}

\begin{proof}
	Let $x_0\in (0,l)$ and $t_0\in (0,T-2l)$ be arbitrary. We set $s:= 2(l - x_0)+t_0$ and consider $\chi \in C_0^\infty(\mathbb{R})$ supported in $(-t_0/2,t_0/2)$. 
	The waves $v_2,v_3$ are given in \eqref{eq_v}:
	\begin{align*}
		&v_2(x,t) = e^{ih^{-1}(t - x - s)} \chi(t - x - s) + o(1),\\
		&v_3(x,t) = \overline{v_2(x,t)},
	\end{align*}
	%According to \eqref{v_h}, for the boundary term given by REF, it follows that
The wave $v_1$ is given by Lemma \ref{go-reflect}:
$$		v_1(x,t) = e^{ih^{-1}(t - x-t_0)} \chi(t - x-t_0)  - e^{ih^{-1}(t + x -2l-t_0)} \chi(t + x -2l-t_0) + o(1).$$
%{Note that the size of the cut-off function $\chi$ needs to be small than $t_0$, in order to satisfy the initial condition of the equation \eqref{eq-wave-nonlinear}.}
The wave $v_0$ is the time-reflection of wave $v$ in \eqref{eq_v}:
	\begin{align*}
		v_0(x,t) =v(x,2l+t_0-t)= e^{ih^{-1}(2l+t_0-t - x)} \chi(2l+t_0-t - x) + o(1),
	\end{align*}
	as $h\rightarrow 0$ uniformly on $(x,t)$. Therefore, since $\chi$ is supported in $(-t_0/2,t_0/2)$, it follows that
	\begin{equation*}
		\int_{0}^{2l+2t_0}\int_0^l av_0v_1v_2v_3 dxdt = - \int_{0}^{2l+2t_0}\int_0^l a(x,t) \chi^2(t - x - s) \chi^2(t + x -2l-t_0)dxdt + o(1).
	\end{equation*}
	Taking $h\rightarrow 0$ and recalling \eqref{a-integral}, we see that 
	\begin{equation*}
		I :=  \int_{0}^{2l+2t_0}\int_0^l a(x,t) \chi^2(t - x - s) \chi^2(t + x -2l-t_0)dxdt
	\end{equation*}
	can be determined by the given $\Lambda_T$. After changing the variables, we obtain
	\begin{equation*}
		I = \frac{1}{2} \int_{-s/2}^{s/2}\int_{-s/2}^{s/2} a\left(\frac{1}{2}(z - y +2l+t_0 - s),\frac{1}{2}(z + y +2l+t_0 +s)\right) \chi^2(y)\chi^2(z)dzdy.
	\end{equation*}
Since the Dirac delta function can be approximated by smooth functions supported in $(-t_0/2,t_0/2)$, it follows that $a\left(\frac{2l+t_0 - s}{2},\frac{2l+t_0 + s}{2}\right)=a\left(x_0,\frac{2l+t_0 + s}{2}\right)$, i.e. the red point in Figure \ref{fig_interval}, can be determined by the given $\Lambda_T$. Hence $a$ on the line $\{x+t=2l+t_0\}$ can be determined.

Let $A\in (2l,T)$. We take $t_0=A-2l>0$ and range $s$ in $[t_0,2l+t_0]$. 
%The size $b$ of the cut-off function $\chi$ can be chosen so that $b<\min\{T-A,A-2l\}$.
The construction above determines $a$ on the line $\{x+t=A\}$.
Taking $A \to T$ determines $a$ on $\{x+t = T\}$ by continuity, and taking $A \to 2l$ determines $a$ on $\{x+t = 2l\}$.
Thus $\Lambda_T$ determines $a$ on the closed parallelogram domain $\{2l \leq x+t \leq T\}$.
\end{proof}

\begin{lemma} \label{star-equal}
For a star graph of equal edge lengths, suppose that control and observation are available on one leaf $z$ (i.e., $\Gamma_0=\{z\}$).
%Dirichlet condition on all other leaves.
Let $e_z$ denote the edge incident to $z$ and $l$ be its length.
Assume $T>2l$.
Then $\Lambda_T$ uniquely determines $a$ on the edge $e_z$ in the domain
$$\Big\{(x,t)\in e_z\times [0,T]: 2l-d(x,z)\leq t\leq T-d(x,z)\Big\}.$$
\end{lemma}

\begin{proof}
    Let $v_S$ be the internal vertex of the star graph. We identify the edge connecting $z$ and $v_S$ with the interval $(0,l)$, and $z$, $v_S$ with $0$, $l$, respectively. We send the waves $v_0$, $v_1$, $v_2$, and $v_3$ as in Lemma \ref{interval}. Then $v_0$, $v_2$, $v_3$ behave the same way in $(0,l)\times (0,2l)$. The only difference is that $v_1$ reflects at $v_S$ with certain coefficient. Namely, by Lemma \ref{CGO-Y},
    \begin{equation}
        v_1(x,t) = e^{ih^{-1}(t - x)} \chi(t - x)  - \frac{n - 2}{n}e^{ih^{-1}(t + x -2l)} \chi(t + x -2l) + o(1)
    \end{equation}
    as $h\rightarrow 0$. Further, we repeat the steps of Lemma \ref{interval} to show that $\Lambda_T$ determines $a$ on the edge connecting $z$ and $v_S$, on the same domain as in Lemma \ref{interval}.
    {Since $x\in e_z$ is identified with $d(x,z)\in [0,l]$ in this parametrization, the domain in Lemma \ref{interval} takes the unparametrized form as the one stated above.}
\end{proof}

\begin{proposition} \label{star-diff}
For any graph, suppose that control and observation are available on a vertex $z$ of degree 1 (i.e., $\Gamma_0=\{z\}$). {Let $e_z$ denote the edge incident to $z$ and $l$ be its length.
Assume $T>2l$.
Then $\Lambda_T$ uniquely determines $a$ on the edge $e_z$ in the same domain as in Lemma \ref{star-equal}.}
%$$\Big\{(x,t)\in e_0\times [0,T]: 2l-d(x,z_0)\leq t\leq T-d(x,z_0)\Big\}.$$
\end{proposition}

Note that in Proposition \ref{star-diff} we do not specify the boundary conditions at other boundary vertices than $z$ because they are irrelevant to our reconstruction process on $e_z$. The boundary conditions can be arbitrary at other boundary vertices than $z$ in Proposition \ref{star-diff}.

\begin{proof}
%\begin{figure}
%    \begin{minipage}{0.42\textwidth}
        %\includegraphics[width=0.4\linewidth]{Star_graph}
        %\caption{Setting of Proposition \ref{star-diff}.}
        %\label{fig_diff}
%    \end{minipage}
%    \begin{minipage}{0.42\textwidth}
%        \includegraphics[width=1\linewidth]{Star_graph2}
%        \caption{Lemma \ref{star-diff}: two intersections}
%        \label{fig_mul2}
%    \end{minipage}
%\end{figure}
Let $\Omega$ be any graph  and $z$ be the one vertex of degree $1$ where we have control and observation. 
Let $y$ be the other endpoint of the edge $e_z$ than $z$.
If 
\begin{equation} \label{eq_interfere}
l=\textrm{Length}(e_z)< 2 \min\Big\{ \textrm{Length}(e): e \textrm{ is incident to } y, \ e\neq e_z \Big\},
\end{equation}
then up to time $2l$, the reflected waves at $V\setminus \{z,y\}$ do not propagate into $e_z$. Hence this situation reduces to Lemma \ref{star-equal}. 

If \eqref{eq_interfere} is not true, then the reflected waves from $V\setminus \{z,y\}$ propagate into $e_z$, as illustrated in grey color in Figure \ref{fig_diff}. 
The edge $e_z$ is parametrized as the interval $[0,l]$ with the boundary vertex $z$ corresponding to $x=0$.
The linear waves $v_0,v_1,v_2,v_3$ are the same as in Lemma \ref{star-equal}.
In the figure, the reflected waves of $v_1$ from vertices other than $z,y$ are in grey above the blue line.
Although there are multiple intersections between blue lines and grey lines, they only intersect with the support of $v_0$ at only one (red) region over $e_z$. 
Indeed, the reflected waves of $v_0$ at vertices other than $z,y$ are located below the $t=x+t_0$ line.
On the other hand, notice that the supports of the waves $v_0,v_1,v_2$ do not intersect on $E\setminus e_z$ up to time $2l$.
This is due to the following observation: the wave $v_1$ vanishes on $E\setminus e_z$ on time interval $[0,l+t_0-b)$, while the wave $v_0$ vanishes on $E\setminus e_z$ on time interval $(l+t_0+b,2l+2t_0]$, where $b$ is the size of the cut-off function in geometric optics.
However, the wave $v_2$ only arrives on $E\setminus e_z$ after time $l+s-b$. This means that if we choose $2b<s-t_0$ in the proof of Lemma \ref{interval}, the wave $v_2$ vanishes on $E\setminus e_z$ on time interval $[0,l+t_0+b] \subset [0,l+s-b)$.
So the left-hand side of the integral \eqref{a-integral} has no contribution from $E\setminus e_z$.
Thus the same argument as in Lemma \ref{star-equal} determines $a$ on the same domain.
\end{proof}

\section{Inverse problem for trees}
\label{sec-general}

To handle the general situation on finite metric trees, we need the following lemmas.

\begin{lemma} \label{star-interfere}
Consider a $3$-edge star graph where control and observation are available on two leaves $z_1,z_2$ (i.e., $\Gamma_0=\{z_1,z_2\}$), and let $z_0$ be the other leaf.
Denote by $e_i$ the edge incident to $z_i$ and by $l_i$ its length.
{Assume $l_1=l_2$ and $T>2(l_1+l_0)$.
Then $\Lambda_T$ uniquely determines $a$ on the domain
$$\big\{(x,t)\in e_j\times [0,T]: 2l_j-d(x,z_j)\leq t\leq T-d(x,z_j)  \big\},\quad j=1,2, \textrm{ and }$$
$$\big\{ (x,t)\in e_0\times [0,T]:  l_1+l_0+d(x,z_0) \leq t \leq T-l_1-l_0+d(x,z_0)\big\}.$$}
\end{lemma}

\begin{proof}
By Proposition \ref{star-diff}, $\Lambda_T$ uniquely determines $a$ on the edges $e_1,e_2$ in the domain stated in Lemma \ref{star-equal}. Hence we focus on the other edge $e_0$.
%where $z_0$ is the other boundary vertex where we do not have control.
%Denote by $e_i$ the edge incident to the boundary vertex $z_i$, $i=0,1,2$,
%and by $l_i$ the length of the edge $e_i$.
%We assume that $l_1=l_2$.
For convenience, we parametrize the edges $e_1$ and $e_0$ together as $[0,l_1+l_0]$, where $z_1$ corresponds to $x=0$; $z_0$ corresponds to $x=l_1+l_0$; $x=l_1$ corresponds to the internal vertex.
We focus on the behavior of waves on $e_0$, i.e., on $[l_1,l_1+l_0]$.

\begin{figure}
    \begin{minipage}{0.47\textwidth}
        \includegraphics[width=0.75\linewidth]{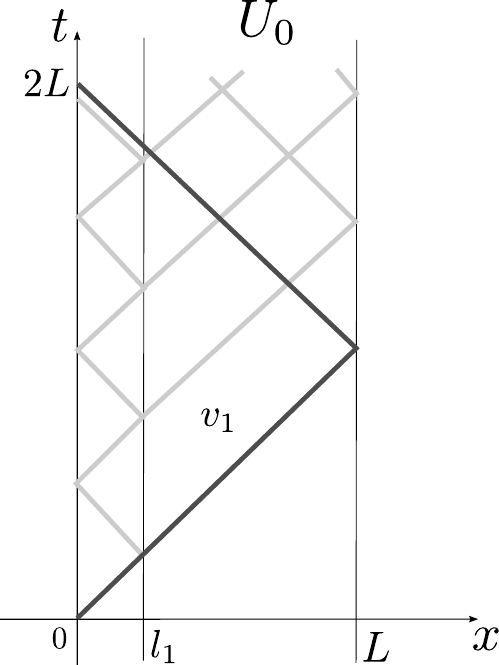}
        \caption{The support of the wave $v_1$ (translated so that the wave starts from the origin).}
        \label{v1}
    \end{minipage}
    \begin{minipage}{0.47\textwidth}
        \includegraphics[width=0.87\linewidth]{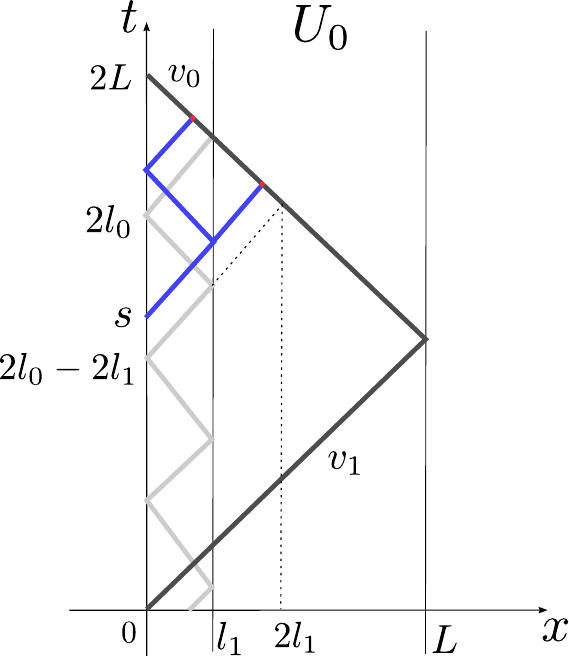}
        \caption{A two-step procedure for reconstructing $a$ on $[l_1,L]$. The blue lines represent the support of $v_2,v_3$.}
       \label{procedure}
    \end{minipage}
\end{figure}

Let $L=l_1+l_0$ so that $T>2L$ by assumption.
As before, we send geometric optics solutions $v_0,v_1,v_2,v_3$ from $z_1$, where $v_1$ is sent at time $t=t_0>0$; $v_2,v_3$ are sent at a delayed time $s>t_0$;
$v_0$ is sent at time $t=2L+t_0$.
What we need to figure out is where these four waves intersect.
{Recall that the requirement of $t_0>0$ is due to the initial condition of the equation \eqref{eq-wave-nonlinear}.
However, for the sake of explaining our construction, it is enough to demonstrate for $t_0=0$.}
We focus on the region over $e_0$, i.e., $$U_0:=[l_1,L]\times \mathbb [0,2L],$$ 
where the second coordinate is the time coordinate.

The support of the wave $v_1$ up to time $2L$ is illustrated in Figure \ref{v1}, where multiple reflections occur. Since we assume $l_1=l_2$, the reflected waves (grey lines) always come into $e_0$ in $2l_1$-increment.
In $U_0$,
$$\bigcap_{b>0}\textrm{supp}(v_1) \cap U_0 \subset \bigcup_{m\in \mathbb{N}} \{t-x=2m l_1\} \cup \{t+x=2L+2ml_1\}.$$
Similarly,
$$\bigcap_{b>0}\textrm{supp}(v_0) \cap U_0 \subset \bigcup_{m\in \mathbb{N}} \{t-x=-2m l_1\} \cup \{t+x=2L-2ml_1\}.$$
Recall that the waves are supported on a thin neighborhood of size $b$, where $b$ is the parameter in the cutoff function, see \eqref{source}.
Although these two supports above intersect frequently, one can see that the intersection is actually two lines union a finite number of points:
\begin{equation}\label{supp-01}
\bigcap_{b>0}\textrm{supp}(v_0) \cap \textrm{supp}(v_1) \cap U_0 =\{t=x\}\cup \{t+x=2L\} \cup \hat{A},
\end{equation}
where $\hat{A}$ is a finite set of points in $U_0$.
This gives much clarity about the intersections of four waves when we take the delayed waves $v_2,v_3$ into account as follows.
%However, the intersection on $[0,l_1]$ can be complicated. For example when $l_1$ divides $L$, the waves $v_0,v_1$ coincide on $[0,l_1]$, and the four waves can intersect at lines when delayed wave $v_2$ is sent at those specific delayed time. For our purpose, one can very well excludes these delayed time and recover them by continuity. In this lemma, one can include those specific delayed time, as the information required is included in the domain of recovery already obtained. However, in general construction, one has to be careful when the construction is done one step further.

\smallskip
First, we consider $v_2,v_3$ with delayed time parameter $s$ close to $2L$. 
If $2l_0=2L-2l_1<s<2L$ and $v_2$ satisfies
\begin{equation}\label{hatA}
\bigcap_{b>0}\textrm{supp}(v_2) \cap \hat{A} =\emptyset,
\end{equation}
%There exists a time $\hat{s}_0 \in (0,2L)$ such that if we send the waves $v_2,v_3$ with delayed time $$s\in (\hat{s}_0,2L), \quad \bigcap_{b>0}\textrm{supp}(v_2) \cap A_{01} =\emptyset,$$
there is no intersection region between $v_0,v_1,v_2,v_3$ in $U_0$ when $b$ is sufficiently small, see Figure \ref{procedure}.
In view of the finiteness of $\hat{A}$, the condition \eqref{hatA} excludes at most a finite number of $s$, since $\bigcap_{b>0} \textrm{supp}(v_2)$ is a finite collection of lines with slopes either $1$ or $-1$.

Next, we lower the time parameter $s$. There exists a time $\hat{s}_1 \in (0,2l_0)$ such that if we send the waves $v_2,v_3$ with delayed time $s\in (\hat{s}_1,2l_0)$ subject to \eqref{hatA},
%$$s\in (\hat{s}_1,2l_2), \quad \bigcap_{b>0}\textrm{supp}(v_2) \cap \hat{A} =\emptyset,$$
then there is only one intersection region between $v_0,v_1,v_2,v_3$ in $U_0$ when $b$ is sufficiently small, see Figure \ref{procedure}.
One can verify that $\hat{s}_1=2l_0-2l_1$ in the present situation, due to our assumption that $l_1=l_2$.
Since we assume to know the graph structure and edge lengths, we know the form of the waves by Lemma \ref{CGO-Y}, in particular, the coefficients depending on how the waves are reflected.
Then following the method in Section \ref{sec-linearization}, one can determine
\begin{equation} \label{lower-s-1}
C  a(L-\frac{s}{2},L+\frac{s}{2}) + \textrm{an integral over }e_1,e_2,
\end{equation}
where $C$ is some known constant. 
This integral over $e_1,e_2$ depends only on the pointwise values of $a$ at some points on $e_1,e_2$ where the four waves intersect.
%Consider if $L/l_1\in \mathbb{Z}$, the $v_1,v_0$ waves completely coincide on $e_1$. so the four waves should have multiple intersections on $e_1$.
On $e_1$, the waves $v_1,v_2$ can intersect only above (including) the line $x+t=2l_1$.
%it is enough to look at these two waves. $v_2$ is sent delayed than $v_1$.
On $e_2$, any intersection occurs only after time $t\geq l_1+d(x,v_S)=l_1+l_2-d(x,z_2)$, $x\in e_2$, where $v_S$ is the internal vertex of the star graph,
because this is the shortest travel time needed for the waves reaching $x\in e_2$ from $z_1$.
%The actual intersection time on $e_2$ is strictly larger than the shortest travel time (except endpoints). The following two situations may happen.
%(1) The intersection only occurs after reflection of $v_1$ on $z_2$. Then the time is at least $l_1+l_2$.
%(2) $v_2$ coincides on $e_2$ with the reflected wave of $v_1$ after hitting $z_0$. In this case, the intersection time is at least $l_1+2l_0+d(x,v_S)$.
On either $e_j$, $j=1,2$, the intersections with the wave $v_0$ must satisfy $t\leq 2L-d(x,z_1)$, $x\in e_j$.
%this is the shortest time to reach $z_1$ at time $2L$.
Recall that $T>2L$ and $d(x,z_1)\geq l_1=l_2\geq d(x,z_2)$ for $x\in e_2$.
Hence, the integral over $e_1,e_2$ in \eqref{lower-s-1} can be computed, since
$a$ is already recovered on $e_1,e_2$ in the domain $\{(x,t)\in e_j\times [0,T]: 2l_j\leq t+d(x,z_j)\leq T\}$ for $j=1,2$ due to Proposition \ref{star-diff}.
%Note that we have used the same parametrization on $e_2$ as $e_1$, namely $z_2$ corresponding to $x=0$.
%The parametrizations on $e_1,e_2$ are exactly the same.
This procedure determines $a$ on 
$$\big\{x+t=2L, \; x\in [l_1, L-\hat{s}_1/2]=[l_1,2l_1]\big\}$$ 
except for a finite number of points excluded by the condition \eqref{hatA}, which determines $a$ on $\{x+t=2L,\, x\in [l_1, 2l_1]\}$ by continuity. 

Note that if $l_1\geq l_0$ the procedure stops after the previous step. Otherwise,
we keep lowering $s$ up to some time $\hat{s}_2$ while avoiding the finite set $\hat{A}$, more precisely $\hat{s}_2=2l_0-4l_1$, and we have two intersection regions in $U_0$ for sufficiently small $b$. Thus one determines the sum
$$C  a(L-\frac{s}{2},L+\frac{s}{2}) + C'  a(x',2L-x') + \textrm{an integral over }e_1,e_2,$$
where $x'=L-s/2-l_1\in [l_1,2l_1]$ in our present situation. Since we have already solved $a$ for $x\in [l_1,2l_1]$, this determines $a$ on the line $\{x+t=2L\}$ for $x\in [2l_1,3l_1]$.
{Repeating this procedure determines $a$ on the line $\{x+t=2L,\, x\in [l_1,L]\}$ in finite steps.}
%Combining with Proposition \ref{star-diff}, we can determine $a$ on $\big\{ x+t=2L,\; x\in [0,L]  \big\}.$

Now instead of sending the wave $v_1$ from $t_0=0$, we vary $t_0$ in $(0,T-2L)$ and repeat the same construction above on the line $\{x+t=2L+t_0\}$. 
Proposition \ref{star-diff} is valid since $2l_1 < 2L+t_0< T$.
Then the same arguments above determine $a$ on 
$\big\{x+t=2L+t_0\in (2L,T), \, x\in [l_1,L] \big\}$, which determines $a$ on $\big\{2L \leq x+t \leq T, \, x\in [l_1,L] \big\}$ by continuity.
The latter domain can be translated to the following unparametrized form:
$$\big\{(x,t)\in e_0\times [0,T]: 
L+d(x,z_0) \leq t \leq T-L+d(x,z_0) \big\},$$
as stated in the lemma.
%Identifying $x$ with $L-d(x,z_0)$.
\end{proof}

Lemma \ref{star-interfere} is still valid when one has control on all but one boundary vertices and the edges incident to these boundary vertices have equal length. The proof stays the same. The reason that the equal length case looks mild is that the trajectories of the reflected wave from those boundary vertices "synchronize".
For the general case, i.e., a star graph with different lengths, the situation looks messier but the approach still works.

\begin{lemma} \label{star-general}
For any star graph, 
suppose that one has control and observation on the boundary vertices $z_1,z_2,\cdots,z_{n}\in \Gamma$, while does not have control and observation on the other boundary vertex $z_0\in \Gamma$. Denote by $e_i$ the edge incident to $z_i$ and by $l_i$ its length.
Let $l_1=\max\limits_{1\leq j\leq n}l_j$ and assume $T>2(l_1+l_0)$.
Then $\Lambda_T$ uniquely determines $a$ on the domain
$$\big\{(x,t)\in e_j\times [0,T]: 2l_j-d(x,z_j)\leq t\leq T-d(x,z_j)  \big\},\quad j=1,2,\cdots,n, \textrm{ and }$$
$$\big\{ (x,t)\in e_0\times [0,T]:  l_1+l_0+d(x,z_0) \leq t \leq T-l_1-l_0+d(x,z_0)\big\}.$$
\end{lemma}

\begin{proof}
Since we have control on $z_1,\cdots,z_{n}$, Lemma \ref{star-diff} gives $a$ on $e_1,\cdots,e_{n}$.
%Lemma \ref{star-interfere} solves the case when $l_1=\cdots=l_{n}$.
We follow the argument in Lemma \ref{star-interfere} to recover $a$ on $e_0$.
We parametrize the edge $e_1,e_0$ together by $[0,l_1+l_0]$ as before, where $z_1$ corresponds to $x=0$.
The waves $v_0,v_1,v_2,v_3$ are sent from $z_1$ in the same way as before, where $v_1$ is sent at time $t_0=0$. 
Denote $L=l_1+l_0$ so that $T>2L$ by assumption.
One can verify that the supports of $v_0,v_1$ still intersect on $e_0$ at two lines union a finite number of points (bigger than $\hat{A}$) as in \eqref{supp-01}.
%The support of $v_1$ on $e_0$ is within the region $\{t-x\geq 0\}\cup \{t+x\geq 2L\}$, the exact lines depending on the graph structure. The support of $v_0$ on $e_0$ is within the region $\{t+x\leq 2L\}\cup \{t-x\leq 0\}$. So these regions only intersect on $e_0$ at two main lines plus a finite number of points. However, the intersecton on $e_1,\cdots,e_n$ can be complicated, i.e., lines.
This means that if we choose the delayed time parameter $2l_0=2L-2l_1<s<2L$ and $v_2$ satisfies \eqref{hatA}, there is no intersection between the waves $v_0,v_1,v_2,v_3$ in $U_0=[l_1,L]\times [0,2L]$, when $b$ is sufficiently small.

Next, we lower the parameter $s$.
The difference between this case and Lemma \ref{star-interfere} is how soon one sees reflected waves on $e_0$. In Lemma \ref{star-interfere}, the reflected waves always come into $e_0$ in $2l_1$-increment due to lengths of $e_1,e_2$ being equal.
In the present case, the reflected waves comes in according to the smallest length, i.e. after time 
$$\hat{s}:= 2\min\{l_1,l_2,\cdots,l_{n}\}.$$
This means that if we choose $\hat{s}_1=2l_0-\hat{s}$ subject to \eqref{hatA}, there is only one intersection region between the four waves in $U_0$ when $b$ is sufficiently small. 
The part of the integral \eqref{a-integral} over the intersection of the waves on $e_j$, for $j=1,\cdots,n$, only depends on the pointwise values of $a$ on some points where the four waves intersect, similar to \eqref{lower-s-1}. 
On $e_1$, the waves $v_1,v_2$ can intersect only above (including) the line $x+t=2l_1$.
On $e_k$, $k=2,\cdots,n$, any intersection occurs only after time $t\geq l_1+d(x,v_S)=l_1+l_k-d(x,z_k)$, $x\in e_k$, where $v_S$ denotes the internal vertex of the star graph.
%see related comments in Lemma 5.1.
On all $e_j$, $j=1,\cdots,n$, the intersection points with the wave $v_0$ must satisfy $t\leq 2L-d(x,z_1)$, $x\in e_j$.
%see similar argument in Lemma 5.1.
Due to the maximal choice of the length $l_1$, for $x\in e_k$, $k=2,\cdots,n$, we see that $l_1+l_k\geq 2l_k$, and $d(x,z_1)\geq l_1\geq l_k\geq d(x,z_k)$.
Hence the required knowledge of $a$ is contained in $\{(x,t)\in e_j\times [0,T]: 2l_j \leq t+d(x,z_j) \leq T\}$ for $j=1,\cdots,n$, which are the domains of recovery obtained in Proposition \ref{star-diff}.
Then the same argument as Lemma \ref{star-interfere} gives $a$ on the line $\{x+t=2L\}$ for $x\in [l_1,L-\hat{s}_1/2]=[l_1,l_1+\hat{s}/2]$. 
The rest of proof is essentially the same as Lemma \ref{star-interfere} and we omit the details.
Repeating this procedure gives $a$ on the line $\{x+t=2L\}$ for $x\in [l_1,L]$ in finite steps due to $\hat{s}>0$.
Ranging $t_0$ in $(0,T-2L)$ recovers $a$ on the domain $\{2L\leq x+t \leq T\}$.
\end{proof}

\begin{lemma} \label{subtree}
In the setting of Lemma \ref{star-general},
let $z_0$ be the boundary vertex where one does not have control and observation.
Then attaching any graph to $z_0$ does not affect the ability to recover $a$ on the edge $e_0$ incident to $z_0$.
\end{lemma}

\begin{proof}
This is simply due to the observation that the left-hand side of the integral \eqref{a-integral} has no contribution from the additionally attached subgraph, with the choice of $v_0,v_1,v_2,v_3$ in Lemma \ref{star-general}.
Recall that we also used this observation in Lemma \ref{star-diff}.
More precisely, we send geometric optics solutions $v_0,v_1,v_2,v_3$ from $z_1$ satisfying $l_1=\max \{l_1,\cdots,l_n\}$, where $v_1$ is sent at time $t=t_0$; $v_2,v_3$ are sent at a delayed time $s>t_0$;
$v_0$ is sent at time $t=2L+t_0$.
On one hand, the attached subgraph does not generate additional intersections on $e_0$, since the reflected wave of $v_1$ from the attached subgraph is only supported above the $x+t=2L+t_0$ line, while the reflected wave of $v_0$ from the attached subgraph is only supported below the $t-x=t_0$ line.
%both are strictly above and below.
On the other hand, the waves do not intersect on the attached subgraph when the size of the cut-off function in geometric optics solutions is small.
The latter is because the wave $v_0$ vanishes on the attached subgraph on time interval $(L+t_0+b,2L+2t_0]$, while the wave $v_2$ vanishes on the attached subgraph on time interval $[0,L+s-b)$, where $b$ is the size of the cut-off function in geometric optics. Thus if we choose $2b<s-t_0$, the supports of $v_0,v_2$ do not intersect on the attached subgraph.
\end{proof}

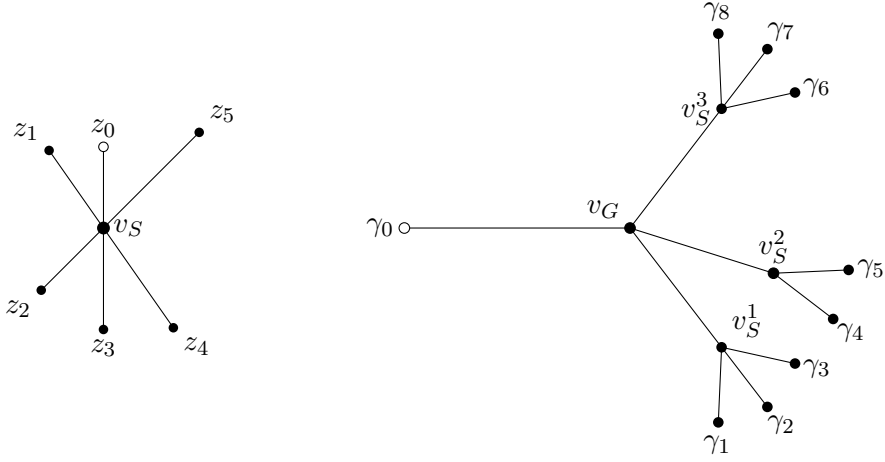
\begin{figure}[h]
%\begin{center}
\centering
 \begin{tikzpicture}[scale=1,grow cyclic,
% every child/.style={fill,circle (2pt)},
 level 1/.style={level distance=3cm,sibling angle=60},
 level 2/.style={level distance=2cm,sibling angle=35},
 level 3/.style={level distance=1cm,sibling angle=40}]
 \coordinate (g1) 
 child { child {child {edge from parent node[pos=1.3] {$\gamma_1$}}  child  {edge from parent node[pos=1.3] {$\gamma_2$}} child {edge from parent node[pos=1.3] {$\gamma_3$}}
 }
         child {child {edge from parent node[pos=1.3]  {$\gamma_4$}} child {edge from parent node[pos=1.3]  {$\gamma_5$}}}
         child[missing]
         child {child {edge from parent node[pos=1.3]  {$\gamma_6$}} child {edge from parent node[pos=1.3]  {$\gamma_7$}} child {edge from parent node[pos=1.3]  {$\gamma_8$}} }};
 \draw[fill=white]  circle (2pt) ;
 \draw node[left] {$\gamma_0$} (g1-1)
 circle (2pt) node[above left] {$v_G$}
        (g1-1-4) node[left] {$v_S^3$}
        (g1-1-2) [fill] circle (2pt) node[above] {$v_S^2$}
        (g1-1-1) node[above right] {$v_S^1$} ;
        \foreach \j in {1,4}
        {\fill (g1-1-\j) circle (2pt);
        \foreach \i in {1,2,3}
        {
        \fill (g1-1-\j-\i) circle (2pt);
        }}
         \foreach \i in {1,2}
                {
                \fill (g1-1-2-\i) circle (2pt);
                }
%\draw (3,-4) node {Figure 7: A tree graph $G$.};

\begin{scope}[xshift=-4cm,scale=0.9]
 \fill (0,0) circle (2.7pt) node[right] {$v_S$};
 \draw[] (0,1.2) node[above] {$z_0$} -- (0,-1.5) node[below] {$z_3$} 
              (225:1.3cm) node[below left] {$z_2$} -- (45:2cm) node[above right] {$z_5$}
              (125:1.4cm) node[above left] {$z_1$} -- (-55:1.8cm) node[below right] {$z_4$};
 \draw[fill=white] (0,1.2) circle (2pt);
\fill       (0,-1.5) circle (2pt)
            (225:1.3cm) circle (2pt)
            (45:2cm) circle (2pt)
            (125:1.4cm) circle (2pt)
                        (-55:1.8cm) circle (2pt);
%\draw (-0.5,-4/1.1) node {Figure 6: A star graph $S$.};
\end{scope}
 \end{tikzpicture}
%\end{center}
\caption{Left: a star graph $S$. Right: a tree graph $G$.} 
\label{star-tree}
\end{figure}

Now we are able to complete the proof of Theorem \ref{main-tree}.
We consider an arbitrary metric tree graph $\Om$ with controls and observations at all but one of the boundary vertices; we will call this vertex a root and denote by $\ga_0$, see an illustration in Figure \ref{star-tree}(right). Let $\Ga_0:=\Ga \setminus \{\ga_0\} =
\{\ga_1,\ldots,\ga_{m-1}\}.$  
%We need the following definition.
 We consider a subgraph of $\Om$ which is a star graph 
consisting of {\sl all} edges incident to an internal vertex $v_S$, see Figure \ref{star-tree}(left). This star graph is called a \emph{sheaf}
if all but one of its edges are adjacent to the boundary vertices in $\Ga_0$.
It is well-known that any tree contains at least one sheaf (see, e.g. \cite{ACLM}).

\begin{proof}[Proof of Theorem \ref{main-tree}]
First, we reconstruct $a$ on each sheaf, i.e., each star subgraph $S$ of $\Omega$ with control and observation on all but one leaves.
Denote by $z_0$ the leaf of a sheaf $S$ where no control and observation are available, and by $v_S$ the internal vertex of $S$, see Figure \ref{star-tree}. 
Let $D_S(z_0)$ denote the maximal distance between $z_0$ and the other leaves $\{z_1,\cdots,z_{n}\}$ of $S$.
Denote by $e_i$ the edge incident to $z_i$ and by $l_i$ its length, for $i=1,\cdots,n$.
To distinguish the role of the edge incident to $z_0$ we denote the edge by $e_S$ here.
Lemma \ref{star-general} and \ref{subtree} determine $a$ on each sheaf in the domains stated in Lemma \ref{star-general}. Namely, the domains of recovery are
\begin{eqnarray} 
\big\{(x,t)\in e_j\times [0,T]: 2l_j-d_S(x,z_j)\leq t\leq T-d_S(x,z_j)  \big\},\; j=1,2,\cdots,n, \label{proof-domain-ej} \\
\big\{ (x,t)\in e_S\times [0,T]:  D_S(z_0)+d_S(x,z_0) \leq t \leq T-D_S(z_0)+d_S(x,z_0)\big\}, \label{proof-domain-e0}
\end{eqnarray}
where $d_S$ denotes the distance function on the star graph $S$.

%While the formulae for the domains of recovery on different edges differ, they can be unified as
%\begin{equation} \label{domain-sheaf}
%\big\{ (x,t)\in S \times [0,T]:  D_S(z_0)+d_S(x,z_0)\leq t \leq T-D_S(z_0)+d_S(x,z_0) \big\},
%\end{equation}
%where $d_S$ denotes the distance function on the star graph $S$.
%To see that this formula is valid on the edges $e_j$ incident to $z_j$ $(j=1,\cdots,n)$ where control and observation is available, one observes that
%$2l_j-d_S(x,z_j)=l_j+d_S(x,v_S) <D_S(z_0)+d_S(x,z_0)$ for any $x\in e_j$, and
%$$T-d_S(x,z_j)=T-(l_j+l_0-d_S(x,z_0)) \geq T-D_S(z_0)+d_S(x,z_0).$$

\smallskip
Next, we consider each subtree $G$ made of sheaves $S_1, \cdots, S_N$ joining together at a common vertex $v_G$, so that the points of no control and observation for each $S_k$ coincide at $v_G$, see Figure \ref{star-tree}.
Suppose that $v_G$ is connected to the leaf $\gamma_0$ where no control and observation is available.
Since we already recovered $a$ on the sheaves $S_1,\cdots S_N$, we only need to recover $a$ on the edge incident to $\gamma_0$.
Let $\gamma_1$ be a leaf of $G$ so that $d_G(\gamma_1,\gamma_0)=D_G(\gamma_0)$, where $D_G=D_G(\gamma_0)$ is the maximal distance between $\gamma_0$ and the other leaves of $G$.
Let us assume $\gamma_1\in S_1$ without loss of generality. 
One can repeat the proof in Lemma \ref{star-general} and \ref{subtree} by sending and observing waves at $\gamma_1$. The only difference is that the edges in Lemma \ref{star-general} are replaced by star graphs. Hence it suffices to verify that the domains of recovery \eqref{proof-domain-ej} and \eqref{proof-domain-e0} for the sheaves are sufficient to complete the reconstruction of $a$ on the edge incident to $\gamma_0$.

We still denote by $e_j$ the edge incident to the leaf $\gamma_j$ and by $l_j$ its length.
%the notation includes $\gamma_0,e_0,l_0$.
We parametrize the path $[\gamma_1 \gamma_0]$ from $\gamma_1$ to $\gamma_0$ together as an interval $[0,D_G]$ with $\gamma_1$ corresponding to $x=0$, and follow the same argument as in Lemma \ref{star-general} by sending waves $v_0,v_1,v_2,v_3$ from $\gamma_1$.
Our remaining task is to check that the intersections of the waves on each sheaf $S_k$ occur within the domains of recovery already obtained above.
On any sheaf $S_k$ for $k\in \{2,\cdots,N\}$, the intersection time of the waves must satisfy 
\begin{equation}
d_G(\gamma_1,v_G)+d_G(x,v_G) \leq t \leq T-d_G(\gamma_1,v_G)-d_G(x,v_G), \quad x\in S_k,\; k\in \{2,\cdots, N\}.
\end{equation}
%This is shortest time that the wave travel from $\gamma_1$ to $x\in S_k$ in other sheaves than $S_1$. The actual time is larger than this, since the wave has to be reflected in order to intersect.
Due to maximal choice of $\gamma_1$, we see that $d_G(\gamma_1,v_G)\geq D_{S_k}(v_G)$, and $2l_j-d_{S_k}(x,\gamma_j)=l_j+d_{S_k}(x,v_{S}^k)< d_G(\gamma_1,v_G)+d_{S_k}(x,v_G)$ for $x\in e_j\subset S_k$, where $v_S^k$ denotes the internal vertex of the star graph $S_k$. 
%$d_G(\gamma_1,v_G)\geq D_{S_k}(v_G)>l_j$ for any $\gamma_j\in S_k$.
Thus, the intersections of waves on $S_k$, $k=2,\cdots,N$, are contained in the domains of recovery \eqref{proof-domain-ej} and $\eqref{proof-domain-e0}$.

\begin{figure}[h]
    \begin{minipage}{0.47\textwidth}   \includegraphics[width=0.84\linewidth, angle=0.3]{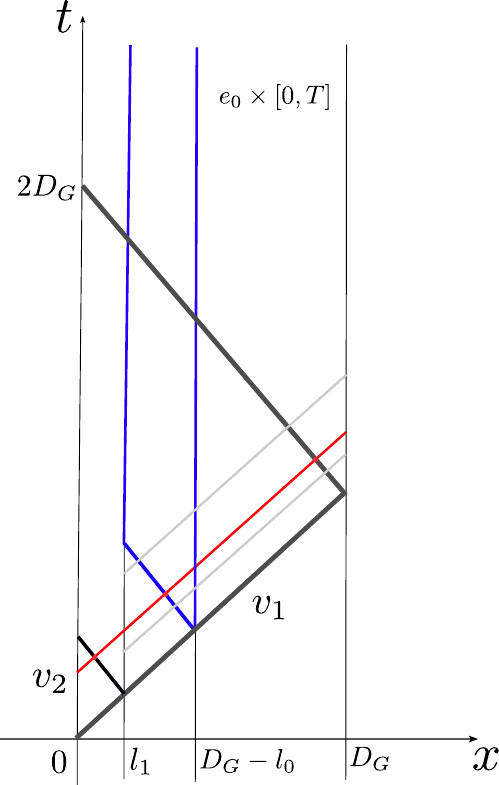}
        \caption{The intersection between the waves $v_1,v_2$ may occur in the middle region earlier than the domain \eqref{second-layer}, the one enclosed by the blue lines, if the wave $v_2$ is sent coinciding with the grey lines that are the reflections of $v_1$ from other leaves in $S_1$ than $\gamma_1$.}
        \label{fig_second-layer}
    \end{minipage}
    \begin{minipage}{0.47\textwidth}
        \includegraphics[width=0.84\linewidth, angle=0.3]{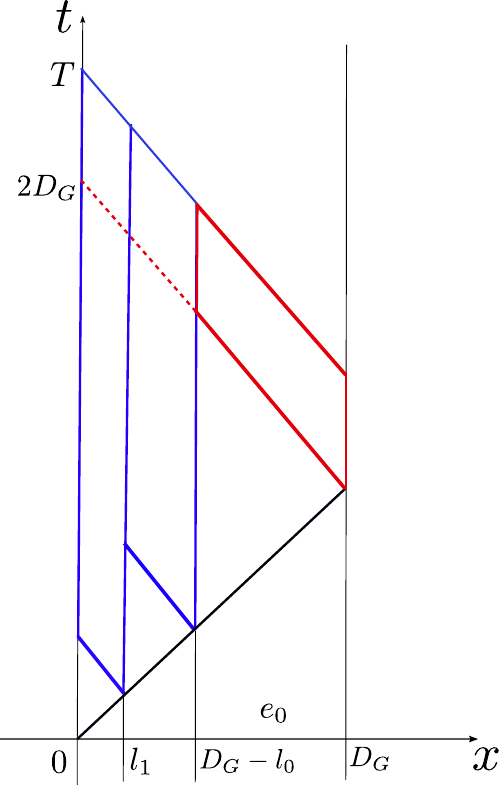}
        \caption{
         A three-step illustration for reconstructing coefficient $a$ on the edge $e_0$ incident to $\gamma_0$ in the domain enclosed by the red lines.
        Enclosed by the blue lines are the domains of recovery for $a$ on edges in previous two steps. 
        The restricted domain of recovery \eqref{Domain_recovery} is indicated by the dashed red line.}
       \label{fig-3-layers}
    \end{minipage}
\end{figure}

On the sheaf $S_1$, the intersection time of the waves on the complement of the path $[\gamma_1 \gamma_0]$ (from $\gamma_1$ to $\gamma_0$) satisfies
\begin{equation}
l_1+d_G(x,v_S^1) \leq t \leq T-l_1-d_G(x,v_S^1), \quad x\in S_1\setminus [\gamma_1 \gamma_0],
\end{equation}
which is clearly contained in the domains of recovery due to the maximal  choice of $\gamma_1$.
On the path $[\gamma_1 \gamma_0]\cap S_1$, the intersection time of the waves satisfies
\begin{equation}
l_1+d_G(x,v_S^1) \leq t\leq T-d_G(x,\gamma_1),\quad x\in e_1,
\end{equation}
%the wave $v_1$ has to first reflect at $v_S^1$ to be able to intersect with $v_2$.
and, up to excluding a finite number of lines of slope $1$,
\begin{equation} \label{second-layer}
d_G(\gamma_1,v_G)+d_G(x,v_G) \leq t \leq T-d_G(x,\gamma_1),\quad x\in S_1\cap [\gamma_1 \gamma_0] \setminus e_1,
\end{equation}
which exactly matches with the domains of recovery \eqref{proof-domain-ej} and $\eqref{proof-domain-e0}$ for $S_1\cap [\gamma_1 \gamma_0]$.
In the latter case, the intersections of the four waves can occur on $S_1\cap [\gamma_1 \gamma_0] \setminus e_1$ earlier than the domain \eqref{second-layer} when a reflection of the wave $v_1$ at other leaves in $S_1$ coincides with the delayed wave $v_2$ on $S_1\cap [\gamma_1 \gamma_0] \setminus e_1$, in which situation the possible intersections are a finite number of (grey) lines of slope 1, see Figure \ref{fig_second-layer}. 
%Namely, when $v_1$ reflected at $\gamma_2$ or $\gamma_3$ coincides with $v_2$ on $e(v_S^1 v_G)$, e.g. when $l_2$ or $l_3$ is small.
These possible intersection lines of slope $1$ correspond to sending the delayed waves $v_2,v_3$ at specific times, which is inconsequential to our construction as we can avoid sending $v_2,v_3$ at those specific times and recover $a$ at the missing points on $e_0\times [0,T]$ by continuity, similar to \eqref{hatA}.
%The support of $v_1$ on $e_0$ is within the region $\{t-x\geq 0\}\cup \{t+x\geq 2D_G\}$, the exact lines depending on the graph structure. The support of $v_0$ on $e_0$ is within the region $\{t+x\leq 2D_G\}\cup \{t-x\leq 0\}$. So these regions only intersect on $e_0$ at two main lines plus a finite number of points, exactly the same as before.
Therefore, repeating the arguments in Lemma \ref{star-general} and \ref{subtree} recovers $a$ on the edge $e_0$ incident to $\gamma_0$ in the domain
\begin{equation} \label{domain-recovery-final}
\big\{ (x,t)\in e_0\times [0,T]:  D_G(\gamma_0)+d_G(x,\gamma_0) \leq t \leq T-D_G(\gamma_0)+d_G(x,\gamma_0)\big\}.
\end{equation}
This construction is illustrated in Figure \ref{fig-3-layers}.
We assume $T>2D_G(\gamma_0)$ so that the domain of recovery \eqref{domain-recovery-final} is nonempty.
Repeating the constructions recovers $a$ on the original tree $\Omega$ in finite steps.
%For general formulation, one considers the path from $\gamma_0$ realizing the largest distance to other leaves, say $\gamma_1$, then send waves from $\gamma_1$. One can parametrize this path by interval $[0,D_G]$, with marker points in the interval representing each vertex on the path. One follows the construction from $x=0$, and recovers the next segment one by one. There would be clearly no issue with the upper time end. For lower time end, the figure shows the domains of recovery match each other exactly, except the situation when the reflection of wave $v_1$ coincides with $v_2$. This generates finite number of intersection lines of slope 1 on known segments, depending on the graph structure. Thus, one only needs to send $v_2$ avoiding all these intersection lines, and recover $a$ on the next segment, except at the missing points corresponding to the intersection lines which can be recovered by continuity. 
%Effectively, recovering $a$ on the whole graph in the restricted domain \eqref{Domain_recovery} is the same as recovering $a$ on the interval $[0,D_G(\gamma_0)]$.

To verify the formulation of the domain \eqref{Domain_recovery} in Theorem \ref{main-tree}, one observes that the same formula \eqref{domain-recovery-final} for each sheaf $S_k$,
$$\big\{ (x,t)\in S_k\times [0,T]:  D_G(\gamma_0)+d_G(x,\gamma_0) \leq t \leq T-D_G(\gamma_0)+d_G(x,\gamma_0)\big\}, \quad k=1,\cdots,N,$$
is a subset of the actual domains of recovery \eqref{proof-domain-ej} and \eqref{proof-domain-e0} for each $S_k$.
Indeed, this is due to $D_G(\gamma_0)-l_0 \geq D_{S_k}(v_G) > l_j$, and $D_G(\gamma_0)-d_G(x,\gamma_0)\geq l_0+D_{S_k}(v_G)-d_G(x,\gamma_0)=D_{S_k}(v_G)-d_G(x,v_G)$, for any $\gamma_j\in S_k$ and $x\in S_k$.
%$v_G$ is the $z_0$ in the formulae.
%When $x\in e_j$ for $\gamma_j\in S_k$, then $D_{S_k}(v_G)-d_G(x,v_G)\geq d(v_G,\gamma_j)-d(x,v_G)=d(x,\gamma_j)$ for any $\gamma_j\in S_k$.
%This is argued for each fixed sheaf $S_k$ and any leaves in the sheaf $S_k$.
Thus, the formulae for the domains of recovery for different edges can be unified as \eqref{Domain_recovery} but at the cost of restricting to smaller domains, see Figure \ref{fig-3-layers}.
\end{proof}

\section{Direct problem} \label{direct}
Let $\Omega=\Omega(V,E)$ be a finite metric tree, and $\Gamma$ be the set of leaves of $\Omega$. 
%Let us introduce a few notations. 
For $T>0$, we set 
\begin{equation*}
    E^0 := C([0,T]; L^2(\Omega)),
    \quad
    E^1 := C([0,T]; H^1(\Omega)) \cap C^1([0,T]; L^2(\Omega)),
\end{equation*}
and equip these spaces with the norms 
\begin{equation*}
    \|u\|_{E^0} = \sup_{t\in [0,T]} \|u(\cdot, t)\|_{L^2(\Omega)},
    \quad
    \|u\|_{E^1} = \sup_{t\in [0,T]} \|u(\cdot, t)\|_{H^1(\Omega)} + \sup_{t\in [0,T]} \|\partial_t u(\cdot, t)\|_{L^2(\Omega)}.
\end{equation*}

\begin{proposition}\label{Prop_linear_energy_est} 
Let $T>0$ and $q\in L^\infty(\Omega)$. Assume that $F\in L^1([0,T]; L^2(\Omega))$ and $f \in H^{3/2}([0,T]; \mathbb{R}^{|\Gamma|})$ satisfies $f = 0$ near $\{t=0\}$ and $\{t=T\}$. Then the boundary value problem 
\begin{equation}\label{lin_eq} 
\begin{cases} 
    (\partial_{t}^2 -\partial_x^2 ) u(x,t)+q(x) u(x,t)=F(x,t), &\text{in }(\Omega\setminus V)\times (0,T),\\ 
    u(\cdot,t)|_{\Gamma}=f(t), & \text{on } \Gamma\times(0,T),\\ 
    u(x,t)|_{t \leq 0}=0, & \text{in } \Omega,\\ 
    u \textrm{ satisfies the Kirchhoff-Neumann condition on } V\setminus\Gamma,
\end{cases} 
\end{equation} has a unique solution $u \in E^1$ satisfying the energy estimate
\begin{equation}\label{est_energy_linear_case}
    \|u\|_{E^1} \leq C\left(\|f\|_{H^{3/2}([0,T]; \mathbb{R}^{|\Gamma|})} + \|F\|_{L^1([0,T]; L^2(\Omega))}\right),
\end{equation}
where $C$ depends only on $T$, $q$, and $\Omega$.

Finally, if assume additionally that 
\begin{equation}\label{add_assumptions}
    F\in C^1([0,T]; L^2(\Omega)),
    \qquad
    F\arrowvert_{t=0}=0,
    \qquad
    \text{and }
    f \in H^{3 + 1/2}([0,T]; \mathbb{R}^{|\Gamma|}),
\end{equation}
then for any edge from $E$, $e\simeq(0,l)$, it follows that $u\arrowvert_{e}(\cdot, t) \in H^2([0,l])$ for all $t\in [0,T]$.
\end{proposition}
%\HOX{M.N.: I required more conditions on $f$}

%\HOX{M.N.: And the estimate is a bit weaker than in \cite[Proposition 1]{LassasLiimatainenPotenciano-MachadoTyni}}

\begin{proof}
    \emph{Step 1: The case $f=0$.} Define a sesquilinear form
    \begin{equation*}
        \mathbf{a}(u,v) = \int_\Omega \partial_x u \overline{\partial_x v} dx,
    \end{equation*}
    with domain $H_0^1(\Omega)$. This yields the nonnegative self-adjoint operator $A$ corresponding to $-\partial_x^2$ with Kirchhoff condition on $V\setminus\Gamma$.

    Consider the space $X = H_0^1(\Omega) \times L^2(\Omega)$, equipped with norm
    \begin{equation*}
        \|(u,v)\|_X = \left(\mathbf{a}(u,u) + \|v\|_{L^2(\Omega)}^2\right)^{1/2}.
    \end{equation*}
    Since $\Gamma \neq \emptyset$, $X$ is a Hilbert space. Define operators 
    \begin{equation*}
        \mathcal{A}_0 = 
        \begin{bmatrix}
            0 & -I\\
            A & 0
        \end{bmatrix},
        \qquad
        \mathcal{B} = 
        \begin{bmatrix}
            0 & 0\\
            q & 0
        \end{bmatrix}
    \end{equation*}
    with the domains $D(\mathcal{A}_0) = D(A)\times H_0^1(\Omega)$ and $D(\mathcal{B}) = X$. Finally, we set $\mathcal{A} = \mathcal{A}_0+ \mathcal{B}$. Since $\mathcal{B}$ is bounded, $D(\mathcal{A}) = D(\mathcal{A}_0)$. Since $-\mathcal{A}$ is a bounded perturbation of the skew-adjoint operator $-\mathcal{A}_0$, it generates strongly continuous semigroup $\{T(t)\}_{t>0}$. Then,
    \begin{equation*}
        \mathcal{U}^F = 
        \begin{bmatrix}
            u^F\\v^F
        \end{bmatrix}
        = \int_0^t T(t - s) \mathcal{F}(s)ds,
        \qquad
        \text{where } 
        \mathcal{F} =
        \begin{bmatrix}
            0\\
            F
        \end{bmatrix},
    \end{equation*}
    solves the equation 
    \begin{equation*}
    \begin{cases}
        \frac{d}{dt} \mathcal{U}^F = - \mathcal{A} \mathcal{U}^F + \mathcal{F},\\
        \mathcal{U}^F(0) = 0.
    \end{cases}
    \end{equation*}
    In particular, $v^F = \partial_t u^F$, and hence, $u^F \in E^1$. Moreover,
    \begin{equation}\label{est_uF}
        \|\mathcal{U}^F\|_{X} \leq 
        \int_0^t \|T(t - s) \mathcal{F}(s)\|_Xds\leq 
        C \int_0^t \|\mathcal{F}(s)\|_Xds = C\|F\|_{L^1([0,T]; L^2(\Omega))}.
    \end{equation}
    Using the Poincare inequality, we complete the proof for the case $f=0$. 

    \emph{Step 2: The case $F=0$.}
    Fix a leaf $v\in\Gamma$ and its incident edge $e\simeq(0,l)$ with endpoint $x=l$ corresponding to $v$. Then there exists a function $w_e\in H^{2}([0,l]_x\times [0,T]_t)$ with
    \begin{equation*}
        w_e = 0 \quad \text{near } \{t=0,T\}\cup \{x=0\}, 
        \qquad 
        w_e\arrowvert_{x=l} = f\arrowvert_v,
    \end{equation*}
    and satisfying 
    \begin{equation}\label{est_w_f}
        \|w_e\|_{H^2([0,l]\times [0,T])} \leq C \|f\arrowvert_v\|_{H^{3/2}([0,T])}.
    \end{equation}
    This can be done by using Sobolev extension theorem and multiplication by a cutoff function which is $1$ near the support of $f\arrowvert_v$ and zero near the boundaries corresponding to $t=0,T$ and $x=0$. 

    The above inequality implies 
    \begin{equation*}
        \sup_{t\in [0,T]} \|w_e(\cdot, t)\|_{H^1(\Omega)} + \sup_{t\in [0,T]} \|\partial_t w_e(\cdot, t)\|_{L^2(\Omega)} \leq C \|f\arrowvert_v\|_{H^{3/2}([0,T])}.
    \end{equation*}
    Set $w=\sum_{v\in\Gamma} w_e$ on boundary edges and $w\equiv0$ elsewhere. Then $w$ satisfies the trivial initial condition, the Kirchhoff-Neumann condition on $V\setminus\Gamma$, and
    \begin{equation}\label{est_for_w}
        \|w\|_{E^1} \leq C \|f\|_{ H^{3/2}([0,T]; \mathbb{R}^{|\Gamma|})}.
    \end{equation}
    We set $u^f = v + w$, where $v$ is the solution to the equation
    \begin{equation*}
    \begin{cases}
        (\partial_t^2-\partial_x^2)v+qv=-(\partial_t^2-\partial_x^2)w - qw, & \text{in }(\Omega\setminus V)\times (0,T),\\
        v=0, & \text{on } \Gamma\times(0,T),\\
        v|_{t\leq 0}=0, & \text{in } \Omega,\\
        v \text{ satisfies the Kirchhoff-Neumann condition on } V\setminus\Gamma.
    \end{cases}
    \end{equation*}
    Since $w_e \in H^2(e\times [0,T])$, we know that 
    \begin{equation*}
        -(\partial_t^2-\partial_x^2)w - qw \in L^2([0,T]; L^2(\Omega)) \subset L^1([0,T]; L^2(\Omega)).
    \end{equation*}
    Therefore, from the first step, we obtain
    \begin{align*}
        \|v\|_{E^1} &\leq C \|(\partial_t^2-\partial_x^2)v+qv\|_{L^1([0,T]; L^2(\Omega))} \\
        & \leq C \sum_{e\in E} \left( \|\partial_t^2 w\|_{L^2(e\times [0,T])} + \|\partial_x^2 w\|_{L^2(e\times [0,T])} + \|q w\|_{L^2(e\times [0,T])}\right)\\
        & \leq C \sum_{e\in E} \|w\|_{H^2(e\times [0,T])}.
    \end{align*}
    Here, the constant $C$ depends only on the graph, $T$, and $q$. By estimate \eqref{est_w_f}, 
    \begin{equation}\label{est_for_v}
        \|v\|_{E^1} \leq C \|f\|_{ H^{3/2}([0,T]; \mathbb{R}^{|\Gamma|})}.
    \end{equation}
    Therefore, due to \eqref{est_for_w}, we derive that
    \begin{equation}\label{est_uf}
        \|u^f\|_{E^1} \leq C \|f\|_{ H^{3/2}([0,T]; \mathbb{R}^{|\Gamma|})}.
    \end{equation}

    \emph{Step 3: General case.}
    Finally, by linearity, $u=u^F+u^f$ solves \eqref{lin_eq} with estimate \eqref{est_energy_linear_case} following from \eqref{est_uF} and \eqref{est_uf}. Uniqueness follows from the group property or an energy identity for homogeneous data.

    \emph{Step 4: Under assumptions \eqref{add_assumptions}.} Since $\mathcal{A}_0$ is skew-adjoint and $\mathcal{B}$ is bounded, we know that $-\mathcal{A} = -\mathcal{A}_0 - \mathcal{B}$ is closed densely defined and satisfies
    \begin{equation*}
        \|(\mathcal{A} + \lambda I)^{-k}\|_{X\rightarrow X} \leq (\lambda - \|\mathcal{B}\|_{X\rightarrow X})^{-k},
        \qquad
        k\in \mathbb{N},
        \text{ }
        \lambda> \|\mathcal{B}\|_{X\rightarrow X}.
    \end{equation*}
    Therefore, by \cite[Theorem 1.19, IX, Ch. 5]{Kato1976} and \cite[(1.49), IX, Ch. 5]{Kato1976}, it follows that $\mathcal{U}^F \in D(\mathcal{A})$, and hence, $u^F\in D(A)$. In particular, $u^F\arrowvert_{e}\in H^2([0,l])$ for any $e\simeq [0,l]$. 

    Since $f \in H^{3 + 1/2}([0,T]; \mathbb{R}^{|\Gamma|})$, we can additionally require $w_e$ to be in $H^{4}([0,l]\times [0,T])$, where $e\simeq [0,l]$. Then,
    \begin{equation*}
        -(\partial_t^2-\partial_x^2)w - qw \in C^1([0,T]; L^2(\Omega))
    \end{equation*}
    and, as we showed above, $v\arrowvert_{e} (\cdot, t) \in H^2([0,l])$, for $e\simeq [0,l]$. By construction of $w$, this is true for $u^f$, and hence, for $u=u^F+u^f$.
\end{proof}

\begin{lemma}\label{lemma_for_direct_problem}
    Let $q\in L^\infty(\Omega)$, $L>0$, and $a\in C^1(\Omega\times [0,T])$ be such that $\|a\|_{C^1(\Omega\times [0,T])} <L$. There exist $\kappa$, $\rho>0$ such that if $f \in H^{3/2}([0,T]; \mathbb{R}^{|\Gamma|})$ satisfies $\|f\|_{ H^{3/2}([0,T]; \mathbb{R}^{|\Gamma|})}<\kappa$ and $f = 0$ near $\{t=0\}\cup\{t=T\}$, then there is a unique solution to 
    \begin{equation}\label{non_linear_eq}
    \begin{cases}
    (\partial_{t}^2 -\partial_x^2 ) u(x,t)+q(x) u(x,t)+a(x,t)u^3(x,t)=0, &x\in \Omega\setminus V, \ t \in (0,T),\\
    u(\cdot,t)|_{\Gamma}=f(t), & t\in (0,T),\\
    u(x,t)|_{t \leq 0}=0, & x\in \Omega, \\
          u \textrm{ subject to the Kirchhoff-Neumann condition on } V\setminus\Gamma,
    \end{cases}
    \end{equation}
    in the ball
    \begin{equation*}
        B_\rho(0) = \{u \in E^1: \; \|u\|_{E^1} \leq \rho\}.
    \end{equation*}
    Moreover, the solution satisfies the estimate
    \begin{equation*}
        \|u\|_{E^1} \leq C_{L,T} \|f\|_{ H^{3/2}([0,T]; \mathbb{R}^{|\Gamma|})}.
    \end{equation*}
\end{lemma}

\begin{proof}
The argument parallels the proof of \cite[Lemma 1]{LassasLiimatainenPotenciano-MachadoTyni}. The only difference is that \cite[Lemma 1]{LassasLiimatainenPotenciano-MachadoTyni} assumes $s \in \mathbb{N}$, whereas here $s=0$. That assumption is used to guarantee that $E^{s+1}$ is an algebra, which remains true in our setting. All other steps are identical.  

Assume $F \in L^1([0,T];L^2(\Omega))$ and $f \in H^{3/2}([0,T];\mathbb{R}^{|\Gamma|})$ be such that $f = 0$ near $\{t=0\}\cup\{t=T\}$. By Proposition \ref{Prop_linear_energy_est}, equation \eqref{lin_eq} admits a unique solution $u \in E^1$.  

For fixed $f$, define the source-to-solution map
\begin{equation*}
   S : E^1 \cap \{F\arrowvert_{t=0} = 0\} \longrightarrow E^1 \cap \{F\arrowvert_{t=0} = 0\},
\end{equation*}
which assigns to $F$ the solution $u$ of equation \eqref{lin_eq}. Next, set
\begin{equation*}
   \Theta : B_\rho(0) \cap \{F\arrowvert_{t=0} = 0\} \longrightarrow B_\rho(0) \cap \{F\arrowvert_{t=0}=0\}, 
   \qquad \Theta u := S(au^3),
\end{equation*}
with $\rho>0$ to be chosen later. Since $E^1$ is an algebra (e.g. \cite{AF}), $\Theta$ is well defined. By \eqref{est_energy_linear_case},
\begin{multline*}
   \|\Theta u\|_{E^1} 
   \leq C_T \bigl(\|f\|_{H^{3/2}([0,T];\mathbb{R}^{|\Gamma|})} + \|au^3\|_{L^1([0,T]; L^2(\Omega))}\bigr) \\
   \leq C_T \bigl(\|f\|_{H^{3/2}([0,T];\mathbb{R}^{|\Gamma|})} + \|au^3\|_{E^0}\bigr) 
   \leq C_T\bigl(\kappa + \|a\|_{C(\Omega\times[0,T])}\rho^3\bigr).
\end{multline*}
Therefore, for sufficiently small $\kappa$ and $\rho$, $\Theta$ defines a self-map of $B_\rho(0)$.

The remaining part of the proof is a repetition of the steps in \cite[Lemma 1]{LassasLiimatainenPotenciano-MachadoTyni}, with \cite[estimate (16)]{LassasLiimatainenPotenciano-MachadoTyni} replaced by \eqref{est_energy_linear_case}. 
\end{proof}

\begin{proposition} \label{prop-diff-epsilon}
Let $q\in L^\infty(\Omega)$, $L>0$, and $a\in C^1(\Omega\times [0,T])$ with $\|a\|_{C^1(\Omega\times [0,T])}<L$. Then there exist $\kappa,\rho>0$ such that for any $f_j\in H^{3/2}([0,T];\mathbb{R}^{|\Gamma|})$ and $\varepsilon_j\in\mathbb{R}$, $j=1,2,3$, satisfying 
\begin{equation*}
    \|\varepsilon_1 f_1+\varepsilon_2 f_2+\varepsilon_3 f_3\|_{H^{3/2}([0,T];\mathbb{R}^{|\Gamma|})}\leq \kappa,
    \qquad 
    f_j = 0 \text{ near } \{t=0\}\cup\{t=T\},
\end{equation*}
the equation \eqref{non_linear_eq} with boundary input $f=\varepsilon_1 f_1+\varepsilon_2 f_2+\varepsilon_3 f_3$ admits a unique solution $u$ in the ball $B_\rho(0)$. Moreover,
\begin{equation*}
    \|u\|_{E^1}\leq C_{L,T}\|f\|_{H^{3/2}([0,T];\mathbb{R}^{|\Gamma|})}.
\end{equation*}

The solution admits the expansion
\begin{equation*}
    u=\varepsilon_1 v_1+\varepsilon_2 v_2+\varepsilon_3 v_3
    +\sum_{\substack{k_1+k_2+k_3=3 \\ k_j \geq 0}}
    \begin{pmatrix}
        3\\k_1,k_2,k_3
    \end{pmatrix}
    \varepsilon_1^{k_1}\varepsilon_2^{k_2}\varepsilon_3^{k_3}
    w_{k_1,k_2,k_3}
    +\mathcal{R},
\end{equation*}
where, for $j=1,2,3$, each $v_j$ solves
\begin{equation*}
\begin{cases}
(\partial_t^2-\partial_x^2)v_j+qv_j=0, & \text{in }(\Omega\setminus V)\times (0,T),\\
v_j=f_j, & \text{on } \Gamma\times(0,T),\\
v_j|_{t\leq 0}=0, & \text{in } \Omega,\\
v_j \text{ satisfies the Kirchhoff-Neumann condition on } V\setminus\Gamma,
\end{cases}
\end{equation*}
and for $k_j\in\{0,1,2,3\}$, each $w_{k_1,k_2,k_3}$ solves
\begin{equation*}
\begin{cases}
(\partial_t^2-\partial_x^2)w_{k_1,k_2,k_3}+q w_{k_1,k_2,k_3}
+av_1^{k_1}v_2^{k_2}v_3^{k_3}=0, & \text{in }(\Omega\setminus V)\times (0,T),\\
w_{k_1,k_2,k_3}=0, & \text{on } \Gamma\times(0,T),\\
w_{k_1,k_2,k_3}|_{t\leq 0}=0, & \text{in } \Omega,\\
w_{k_1,k_2,k_3} \text{ satisfies the Kirchhoff-Neumann condition on } V\setminus\Gamma,
\end{cases}
\end{equation*}
with remainder bounded by
\begin{equation*}
    \|\mathcal{R}\|_{E^1}\leq C_T \|a\|_{E^1}^2  \|f\|_{H^{3/2}([0,T];\mathbb{R}^{|\Gamma|})}^5.
\end{equation*}
\end{proposition}

\begin{proof}
The argument follows from \cite[Proposition 2]{LassasLiimatainenPotenciano-MachadoTyni}.  
Existence, uniqueness, and the first estimate are given by Lemma \ref{lemma_for_direct_problem}.  
We now prove the remaining claims.  

Note that $\mathcal{F} = u - \varepsilon_1 v_1 - \varepsilon_2 v_2 - \varepsilon_3 v_3$ satisfies
\begin{equation*}
\begin{cases}
(\partial_{t}^2 -\partial_x^2 ) \mathcal{F}+q\,\mathcal{F}
=-au^3, &\text{in } (\Omega\setminus V)\times (0,T),\\
\mathcal{F}\arrowvert_{\Gamma}=0, & \text{on }\Gamma\times(0,T),\\
\mathcal{F}\arrowvert_{t \leq 0}=0, & \text{in }\Omega, \\
\mathcal{F} \text{ subject to the Kirchhoff-Neumann condition on } V\setminus\Gamma.
\end{cases}
\end{equation*}
By Proposition \ref{Prop_linear_energy_est} and the algebra property of $E^1$, we obtain
\begin{multline}\label{est_F}
    \|\mathcal{F}\|_{E^1} 
    \leq   C_T \|au^3\|_{L^1([0,T]; L^2(\Omega))} \leq \sup_{t\in [0,T]} \|au^3(\cdot, t)\|_{H^1(\Omega)}\\
    \leq C_T \|au^3\|_{E^1}
    \leq C_T \|a\|_{E^1}\|u\|_{E^1}^3
    \leq C_T \|a\|_{E^1}\|f\|_{H^{3/2}([0,T];\mathbb{R}^{|\Gamma|})}^3.
\end{multline}

Next, consider the remainder term
\begin{equation*}
    \mathcal{R}
    = \mathcal{F}
    - \sum_{\substack{k_1+k_2+k_3=3 \\ k_j \geq 0}}
    \begin{pmatrix} 3\\ k_1,k_2,k_3 \end{pmatrix}
    \varepsilon_1^{k_1}\varepsilon_2^{k_2}\varepsilon_3^{k_3}
    w_{k_1,k_2,k_3}.
\end{equation*}
It satisfies
\begin{equation}\label{eq_R}
\begin{cases}
(\partial_{t}^2 -\partial_x^2 ) \mathcal{R}+q(x)\mathcal{R}
= -au^3 + a\big(\varepsilon_1 v_1+\varepsilon_2 v_2+\varepsilon_3 v_3\big)^3, & \text{in }(\Omega\setminus V)\times (0,T),\\
\mathcal{R}\arrowvert_{\Gamma}=0, & \textrm{on }\Gamma\times(0,T),\\
\mathcal{R}\arrowvert_{t \leq 0}=0, & \textrm{in }\Omega, \\
\mathcal{R} \text{ subject to the Kirchhoff-Neumann condition on } V\setminus\Gamma.
\end{cases}
\end{equation}
Applying Proposition \ref{Prop_linear_energy_est}, estimate \eqref{est_F}, and the algebra property of $E^1$, we obtain
\begin{align}\label{arg_est_R}
    \nonumber\|\mathcal{R}\|_{E^1} 
    &\leq C_T \|au^3 + a(\varepsilon_1 v_1+\varepsilon_2 v_2+\varepsilon_3 v_3)^3\|_{E^1}\\
    &\leq C_T \|a\|_{E^1}\|\mathcal{F}\|_{E^1}
        \|u^2 + u(\varepsilon_1 v_1+\varepsilon_2 v_2+\varepsilon_3 v_3) + (\varepsilon_1 v_1+\varepsilon_2 v_2+\varepsilon_3 v_3)^2\|_{E^1}\\
    \nonumber&\leq C_T \|a\|_{E^1}^2 \|f\|_{H^{3/2}([0,T];\mathbb{R}^{|\Gamma|})}^5.
\end{align}
This completes the proof.
\end{proof}

The following proposition justifies the integral identity \eqref{a-integral} used in Section \ref{sec-linearization}.

\begin{proposition}\label{prop-linearization-identity}
Let the assumptions and notation of Proposition~\ref{prop-diff-epsilon} hold. 
Assume additionally that $f_j \in H^{3+1/2}([0,T]; \mathbb{R}^{|\Gamma|})$ for $j=1,2,3$.
Let $h\in H^{3 + 1/2}([0,T]; \mathbb{R}^{|\Gamma|})$ be compactly supported in $(0,T)$ with sufficiently small $H^{3/2}$ norm, and let $v^h$ be the solution of 
\begin{equation*}
    \begin{cases}
    (\partial_{t}^2 -\partial_x^2 ) v(x,t)+q(x) v(x,t)=0, &x\in \Omega\setminus V, t \in (0,T),\\
    v(\cdot,t)|_{\Gamma}=h(t), & t\in (0,T),\\
    v(x,t)|_{t \geq T}=0, & x\in \Omega, \\
    v \textrm{ subject to the Kirchhoff--Neumann condition on } V\setminus\Gamma.
    \end{cases}
\end{equation*}
Then 
\begin{equation}\label{eq-linearization-identity}
    - \int_0^T \int_{\Gamma} h \partial_{\varepsilon_1} \partial_{\varepsilon_2}\partial_{\varepsilon_3}\Lambda(\varepsilon_1f_1 + \varepsilon_2f_2 + \varepsilon_3f_3)\big|_{\varepsilon_1=\varepsilon_2=\varepsilon_3 = 0} = 6\int_0^T\int_\Omega a v^h v_1 v_2 v_3.
\end{equation}
\end{proposition}

\begin{proof}
We recall the finite difference operator: for a function $g$ depending on $\varepsilon = (\varepsilon_1,\varepsilon_2,\varepsilon_3)$,
\begin{equation*}
    D^3_{\varepsilon_1,\varepsilon_2,\varepsilon_3}\big|_{\varepsilon = 0}  g(\varepsilon_1,\varepsilon_2,\varepsilon_3)
    = \frac{1}{\varepsilon_1\varepsilon_2\varepsilon_3}\sum_{\sigma\in\{0,1\}^3} (-1)^{3+|\sigma|} g(\sigma_1 \varepsilon_1,\sigma_2 \varepsilon_2,\sigma_3 \varepsilon_3).
\end{equation*}
By \cite[Lemma 10]{LassasLiimatainenPotenciano-MachadoTyni} and the expansion in Proposition~\ref{prop-diff-epsilon},
\begin{equation*}
    D^3_{\varepsilon_1,\varepsilon_2,\varepsilon_3} (\partial_{t}^2 -\partial_x^2 + q)u = 6  (\partial_{t}^2 -\partial_x^2 + q) w_{1,1,1} + D^3_{\varepsilon_1,\varepsilon_2,\varepsilon_3} (\partial_{t}^2 -\partial_x^2 + q)\mathcal{R}.
\end{equation*}
Integration by parts against $v^h$ gives
\begin{multline*}
     - \int_0^T \int_{\Gamma} v^h D^3_{\varepsilon_1,\varepsilon_2,\varepsilon_3} \Lambda(\varepsilon_1f_1 + \varepsilon_2f_2 + \varepsilon_3f_3)\\
     = 6\int_0^T\int_\Omega v^h (\partial_{t}^2 -\partial_x^2 + q) w_{1,1,1} + \int_0^T\int_\Omega v^h  D^3_{\varepsilon_1,\varepsilon_2,\varepsilon_3} (\partial_{t}^2 -\partial_x^2 + q)\mathcal{R}.
\end{multline*}
By recalling that $(\partial_t^2 - \partial_x^2 + q) w_{1,1,1} = -a v_1 v_2 v_3$ and integrating by parts once more, we obtain
\begin{multline*}
     - \int_0^T \int_{\Gamma} h  D^3_{\varepsilon_1,\varepsilon_2,\varepsilon_3} \Lambda(\varepsilon_1f_1 + \varepsilon_2f_2 + \varepsilon_3f_3)\\
     = 6\int_0^T\int_\Omega a  v^h v_1 v_2 v_3 + \int_0^T\int_\Omega v^h  D^3_{\varepsilon_1,\varepsilon_2,\varepsilon_3} (\partial_{t}^2 -\partial_x^2 + q)\mathcal{R}.
\end{multline*}
It remains to show that the remainder term vanishes in the limit $\varepsilon_j \to 0$.
Since 
\begin{equation*}
    (\partial_{t}^2 -\partial_x^2 + q)\mathcal{R} = -au^3 + a(\varepsilon_1 v_1+\varepsilon_2 v_2+\varepsilon_3 v_3)^3,
\end{equation*}
the estimate \eqref{arg_est_R} implies
\begin{equation*}
    \big\|D^3_{\varepsilon_1,\varepsilon_2,\varepsilon_3} (\partial_{t}^2 -\partial_x^2 + q)\mathcal{R}\big\|_{E^0} \leq   \frac{C_T}{\varepsilon_1\varepsilon_2\varepsilon_3} \|a\|_{E^1}^2 \|\varepsilon_1f_1 + \varepsilon_2f_2 + \varepsilon_3f_3\|_{H^{3/2}([0,T];\mathbb{R}^{|\Gamma|})}^5.
\end{equation*}
Therefore, taking $\varepsilon_j \rightarrow 0$, $j=1,2,3$, we obtain \eqref{eq-linearization-identity}.
\end{proof}

\section{Acknowledgments}
The research of S.A. was  supported  in part by the National Science Foundation, grant DMS 2308377, and by the Ministry of Education and Science of the Russian Federations part of the program of the Moscow Center for Fundamental and Applied Mathematics under the Agreement No. 075-15-2025-345. The most part of this work was done during the visits of S.A. to the University of Helsinki. He is very grateful to the Department of Mathematics and Statistics of the University of Helsinki for its hospitality.
M.L. and J.L. were supported by the PDE-Inverse project of the European Research Council of the European Union, project 
101097198, and the Research Council of Finland, grants 273979 and 284715. L.O. and M.N. were supported by the European Research Council of the European Union, grant 101086697 (LoCal), and the Research Council of Finland, grants 347715, 353096 (Centre of Excellence of Inverse Modelling and Imaging) and 359182 (Flagship of Advanced Mathematics for Sensing Imaging and Modelling). Views and opinions expressed are those of the authors only and do not necessarily reflect those of the European Union or the other funding organizations.

\end{document}